\documentclass[10pt,a4paper]{paper}
\usepackage[T1]{fontenc}
\usepackage{lmodern}
\usepackage{amsmath}
\usepackage{amssymb}
\usepackage{amscd}
\usepackage{color}
\usepackage{verbatim}

\usepackage{amsmath,amsfonts,amsthm,mathtools}
\usepackage{pgf,pgfplots,tikz,pgfplotstable}
\usetikzlibrary{arrows,chains,matrix,positioning,scopes,shapes.callouts,fit,calc}
\usepackage{booktabs}
\usepackage[autolanguage]{numprint}
\usepackage[linesnumbered,ruled]{algorithm2e}
\usepackage{listings}
\renewcommand{\P}{{\rm P}}

\newcommand{\diag}{\mathrm{diag}}

\newcommand{\ue}{\mathrm{e}}  
\newcommand{\ud}{\, \mathrm{d}}  

\newcommand{\m}[1]{\begin{bmatrix} #1 \end{bmatrix}}

\newtheorem{defn}{Definition}[section]
\newtheorem{lem}[defn]{Lemma}
\newtheorem{prop}[defn]{Proposition}
\newtheorem{theo}[defn]{Theorem}
\newtheorem{thm}[defn]{Theorem}
\newtheorem{cor}[defn]{Corollary}
\theoremstyle{remark}
\newtheorem{rem}[defn]{Remark}

\newcommand{\bs}{\boldsymbol}

\renewcommand{\m}[1]{\begin{bmatrix}#1\end{bmatrix}}
\DeclarePairedDelimiter{\abs}{\lvert}{\rvert}
\DeclarePairedDelimiter{\norm}{\lVert}{\rVert}
\DeclareMathOperator{\offdiag}{offdiag}

\newcommand{\leqdot}{\mathrel{\dot{\leq}}}
\newcommand{\ones}{\mathbf{1}}
\renewcommand{\mp}{\mathrm{u}}

\tikzset{state/.style={semithick,fill=blue,draw=none,text=white,circle}}
\tikzset{semithick,censored state/.style={state,fill=red}}
\tikzset{states/.style={state,fill=blue!30!white,text=black}}
\tikzset{semithick,censored states/.style={states,fill=red!30!white}}
\tikzset{prob/.style={semithick,->,>=stealth',auto}}

\title{Componentwise accurate fluid queue computations using doubling algorithms}
\author{Giang T. Nguyen \thanks{The University of Adelaide, School of Mathematical Sciences, \texttt{giang.nguyen@adelaide.edu.au}} and Federico Poloni\thanks{The University of Pisa, Computer Science Department, \texttt{fpoloni@di.unipi.it}}
}

\usepackage{hyperref}
\begin{document} 
\maketitle

\begin{abstract}
Markov-modulated fluid queues are popular stochastic processes frequently used for modelling real-life applications. An important performance measure to evaluate in these applications is their steady-state behaviour, which is determined by the stationary density. Computing it requires solving a (nonsymmetric) M-matrix algebraic Riccati equation, and indeed computing the stationary density is the most important application of this class of equations.

Xue \emph{et al.} (\emph{Accurate solutions of M-matrix algebraic Riccati equations}, Numer.~Math., 2012) provided a componentwise first-order perturbation analysis of this equation, proving that the solution can be computed to high relative accuracy even in the smallest entries, and suggested several algorithms for computing it. An important step in all proposed algorithms is using so-called \emph{triplet representations}, which are special representations for M-matrices that allow for a high-accuracy variant of Gaussian elimination, the GTH-like algorithm. However, triplet representations for all the M-matrices needed in the algorithm were not found explicitly. This can lead to an accuracy loss that prevents the algorithms from converging in the componentwise sense.

In this paper, we focus on the \emph{structured doubling algorithm}, the most efficient among the proposed methods in Xue \emph{et al.}, and build upon their results, providing (i) explicit and cancellation-free expressions for the needed triplet representations, allowing the algorithm to be performed in a really cancellation-free fashion; (ii) an algorithm to evaluate the final part of the computation to obtain the stationary density; and (iii) a componentwise error analysis for the resulting algorithm, the first explicit one for this class of algorithms. We also present numerical results to illustrate the accuracy advantage of our method over standard (normwise-accurate) algorithms.
\end{abstract} 
\section{Introduction} 
	\label{sec:intro} 
Markov-modulated fluid queues $\{X(t), \varphi(t): t \geq 0\}$ are two-dimensional Markov processes frequently used for modeling real-life applications, such as risk theory, actuarial science, environmental systems and telecommunication networks; see for instance \cite{BOS,LT,MMS,vFMS}. In a fluid queue, the variable $X(t)$, called \emph{level}, varies linearly at rate $c_i$ whenever the Markov chain $\varphi(t)$ is in state $i$. 


The key component for obtaining the stationary distribution of a Markov-modulated fluid queue, which determines its steady-state behaviour, is finding the minimal nonnegative solution to a \emph{nonsymmetric algebraic Riccati equation} (NARE) of the form 
%
\begin{align} 
	\label{eqn:nare} 
	B - A X - X D + X C X = 0,
\end{align}
Several algorithms have been suggested for solving equations of this kind. We refer the reader to the review article \cite{bimp} and the book \cite{bim} for a list of the possibilities. In our problem, the matrix
\[
 M:= \left[\begin{array}{rr} D & -C \\ -B & A \end{array}\right]
\]
is a singular irreducible M-matrix, satisfying $M\ones=\bs{0}$, where $\ones$ is the vector with all components $1$ and $\bs{0}$ the one with all components $0$. The linear algebra literature has focused on a larger class of equations, so-called \emph{M-matrix algebraic Riccati equations}, in which $M$ is either nonsingular, or singular and irreducible M-matrix. 
Nevertheless, fluid queues are the main application of this class of equations, having a broad applicability and considerable interest in applied probability. 
Hence, we 
focus in this paper on the case of NAREs appearing in fluid queues, adopting the notation and terminology closer to the probabilistic meaning of these equations.

We consider here the \emph{structured doubling algorithm} (SDA), introduced in Guo \emph{et al.}~\cite{glx05}, and two variants known as SDA-ss~\cite{bmp} and ADDA~\cite{wwl}. 
These algorithms are the most efficient ones known for this class of equations \cite{bimp}.
We develop a version of these algorithms that computes componentwise accurate results with provable bounds. 
%
An approximation $\tilde{Q}$ of a matrix $Q$ is said to be within a componentwise error threshold $\varepsilon>0$ if it satisfies the inequality 
\begin{equation} 
	\label{eqn:componentwise_error}
		|\tilde{Q}_{ij} - Q_{ij}| \leq \varepsilon |Q_{ij}| \quad \mbox{for all } i \mbox{ and } j.
\end{equation}
The threshold $\varepsilon$ may include a condition number for the problem, or a moderately-growing function of the matrix size. Normally, linear algebra algorithms focus on normwise stability instead, that is, on error bounds of the kind
$
 		\norm{\tilde{Q}-Q} \leq \varepsilon \norm{Q},
$  with $\norm{\cdot}$ being, for instance, the Euclidean norm.
%
%
%
The main difference between the two, is that a componentwise accurate algorithm ensures  not only the dominating components but also the smallest entries in the matrices are computed with a high number of significant digits. This additional accuracy is especially useful for problems whose solutions are badly scaled. In particular, having componentwise error bounds is important in probabilistic applications: some of the states of the Markov chain $\varphi(t)$ may represent component failures or other exceptional behaviours; thus, very small entries are expected and it is crucial that they are computed with high accuracy. 

A componentwise perturbation theory for equation \eqref{eqn:nare} has recently appeared in~\cite{xxl12}; its authors in addition described a first attempt to obtain a componentwise accurate algorithm. Our main improvement to their algorithm is showing how to update special representations of the M-matrices appearing in the algorithm, known as \emph{triplet representations}. We refer the reader to Section~\ref{sec:cfSDA} for definitions and more detail.

Moreover, we provide a componentwise accurate implementation for the part that follows the solution of \eqref{eqn:nare}, that is, the computation of the stationary density of the queue using this solution. This part is usually neglected in the linear algebra literature, but it is nevertheless required to complete the computation of interest in queuing theory. 

We present numerical experiments to demonstrate that updating correctly the triplet representations is crucial to attain componentwise accuracy. In addition, we prove in the Appendix an explicit componentwise error bound for the computed solution $\Psi$, the first one to our knowledge for this class of algorithms.

We organize the paper as follows. In Section~\ref{sec:statdis}, we formally describe the Markov-modulated fluid queue and present a formula for its stationary distribution. In Section~\ref{sec:SDA}, we introduce the structure-preserving doubling algorithm. We 
describe our componentwise-accurate variant in Section~\ref{sec:cfSDA}. Numerical experiments are presented in Section~\ref{sec:ne}, and a stability proof is included in the Appendix. 

\section{The stationary density of a fluid queue} 
	\label{sec:statdis}	
Let $T$ be the generator matrix of a continuous-time Markov chain (CTMC) $\varphi(t)$ on the finite state space $\mathcal{S}$; that is, $T\ones =\bs{0}$, $T_{ij}\geq 0$ for $i\neq j$ and $i, j \in \mathcal{S}$. We assume that $T$ is irreducible. A Markov-modulated fluid queue with a reflecting boundary at $0$ \cite{ram99,mitra,roger94,soares05} is a two-dimensional Markov process $\{X(t),\varphi(t)\}$, where the \emph{level process} $X(t)\geq 0$ varies linearly at rate $c_i$ whenever the \emph{phase process} $\varphi(t)$ is in state $i \in \mathcal{S}$. In other words, 
\begin{align*}  
\frac{\ud}{\ud t} X(t) & =\begin{cases} c_i  & \mbox{ for } \varphi(t) = i \in \mathcal{S} \mbox{ and } X(t) > 0, \\
                                \max\{0,c_i\} & \mbox{ for }  \varphi(t) = i \in \mathcal{S} \mbox{ and } X(t) = 0.\end{cases}
\end{align*} 
We define $C = \mbox{diag}(c_i)_{i \in \mathcal{S}}$ to be the diagonal rate matrix of the fluid level, and assume that $c_i \neq 0$ for all $i \in \mathcal{S}$. This assumption is without loss of generality, as one can analyze models with zero rates by converting them into models without zero rates using censoring techniques~\cite[Section~1.7]{soares05} which are componentwise-accurate without modification. Moreover, we assume that the fluid queue is \emph{positive recurrent}; this is equivalent to requiring $\boldsymbol{\xi}C\boldsymbol{1} < 0$, where $\boldsymbol{\xi} \geq \boldsymbol{0}$ is the stationary distribution row vector of the Markov chain $\varphi(t)$ ($\boldsymbol{\xi} T = \boldsymbol{0}$ and $\boldsymbol{\xi}\boldsymbol{1} = 1$). This is a necessary condition for the existence of our quantity of interest, the stationary density.

We decompose the state space $\mathcal{S}$ into two exhaustive and disjoint subsets: $\mathcal{S}_{+} := \{i \in \mathcal{S}: c_i > 0\}$ and $\mathcal{S}_{-} := \{i \in \mathcal{S}: c_i < 0\}$, and define $n := |\mathcal{S}|$, $n_{+} := |\mathcal{S}_{+}|$, and $n_{-} := |\mathcal{S}_{-}|$. Assume without loss of generality (up to reordering) that $\mathcal{S}_+=\{1,2,\dots,n_+\}$; then, the matrices $T$ and $C$ can be written as  
\begin{align*}
T := \left[\begin{array}{cc} 
T_{++} & T_{+-} \\ 
T_{-+}  & T_{--} 
\end{array}\right], \quad
C := \left[\begin{array}{cc} 
C_{+} & 0 \\
0 & C_{-} 
\end{array}\right],
\end{align*} 
where $T_{\ell m}$ contains entries of $T_{ij}$ for $i \in \mathcal{S}_{\ell}$ and $j \in \mathcal{S}_m$, and $C_{\ell} := \diag(c_i)_{i \in \mathcal{S}_{\ell}}$. Next, we denote by $\boldsymbol{F}(x) := (F_i(x))_{i \in \mathcal{S}}$ and $\boldsymbol{p} := (p_i)_{i \in \mathcal{S}}$, respectively, the stationary distribution row vector and the probability mass row vector at level zero of $\{X(t), \varphi(t)\}$, that is,
\begin{align*}
F_i(x) & := \lim_{t \rightarrow \infty} \P[X(t) \leq x, \varphi(t) = i] \quad \mbox{for } x > 0 \mbox{ and } i \in \mathcal{S}, \\
p_i & := \lim_{t \rightarrow \infty} \P[X(t) = 0, \varphi(t) = i] \quad \mbox{for } i \in \mathcal{S}, 
%
\intertext{and denote by $\boldsymbol{f}(x) := (f_i)_{i \in \mathcal{S}}$ the associated density row vector, with} 
%
f_i(x) & := \frac{\ud}{\ud x}F_i(x)\quad \mbox{for } x > 0 \mbox{ and } i \in \mathcal{S}.
\end{align*} 
Based on the physical interpretations of the rates $c_i$, we immediately obtain $\bs{p} = (\bs{0}, \bs{p}_{-})$, where $\bs{p}_{-} = (p_i)_{i \in \mathcal{S}_{-}}$. Furthermore, it is well-established \cite{ram99} that the stationary density vector $\boldsymbol{f}(x)$ satisfies the following differential equation 
\begin{align} 
\frac{\ud}{\ud x}\boldsymbol{f}(x) C = \boldsymbol{f}(x)T,  \label{eqn:depi}
\end{align} 
and is given by 
\begin{align}
	\label{realinv} 
\boldsymbol{f}(x) & = \boldsymbol{p}_-T_{-+} \ue^{Kx}V \quad \mbox{for } x > 0, 
%
\intertext{with} 
\ue^{Kx} & = \sum_{i = 0}^{\infty} (Kx)^i/i!,  \nonumber \\
%
K & = C_+^{-1}T_{++} + \Psi |C_-|^{-1}T_{-+},  \label{eqn:K}\\
V & = \m{C_+^{-1} & \Psi \abs{C_-}^{-1}}, \label{eqn:V}
\end{align} 
the vector $\bs{p}_{-}$ being the unique solution to the system of equations 
\begin{align}
 \bs{p_-} (T_{--} + T_{-+} \Psi) &= \bs{0},  \label{eqn:invprob1} \\
  \bs{p_-}\left(\mathbf{1}-T_{-+}K^{-1}V\mathbf{1}\right) &= 1. \label{eqn:invprob2}
\end{align}
and $\Psi$ the minimal nonnegative solution to the nonsymmetric algebraic Riccati equation
\begin{align}
	\label{eqn:Psi} 
\Psi |C_{-}^{-1}| T_{-+}\Psi + C_{+}^{-1}T_{++} \Psi + \Psi |C_{-}^{-1}|T_{--} + C_{+}^{-1}T_{+-} = 0.
\end{align}
The matrix $\Psi$ is stochastic and has a probabilistic meaning: each entry $\Psi_{ij}$ is the probability of the fluid returning, from above, to the initial level $x \geq 0$ in phase~$j\in \mathcal{S}_{-}$, after starting in phase~$i \in \mathcal{S}_{+}$ and avoiding all levels below~$x$ \cite[Theorem~2.1]{glr11b}.


In an alternative linear algebra view, the structure-preserving doubling algorithm determines a basis $\m{I_{n_{+}} & -\Psi}$ for the antistable left invariant subspace of the matrix $C^{-1}T$; that is, there exists a matrix---which turns out to be $K$ as defined in \eqref{eqn:K}---of order $n_{+} \times n_{+}$ with all its eigenvalues in the right half-plane such that 
\begin{align}
	\label{eqn:leftinv}
\m{I & -\Psi} C^{-1}T = K\m{I & -\Psi}. 
\end{align} 

The formulas \eqref{realinv}--\eqref{eqn:V} show that solving the Riccati equation~\eqref{eqn:Psi} for $\Psi$ and calculating the matrix-exponential $\ue^{Kx}$ are two main tasks in the evaluation of the stationary density $\boldsymbol{f}(x)$.

\section{Structure-preserving doubling algorithms} 
	\label{sec:SDA}
	

Guo \emph{et al.}~\cite{glx05} developed a structure-preserving doubling algorithm for solving nonsymmetric algebraic Riccati equation, and proved that convergence to the minimal nonnegative solution is monotonically increasing and quadratic. Later, Bini \emph{et al.}~\cite{bmp} and Wang \emph{et al.}~\cite{wwl} presented two variants, SDA-ss and ADDA respectively, which differ only in the initialization step and have faster convergence speeds. We describe below the algorithms, starting from the most recent and general one, the ADDA \cite{wwl}.

We first define a map $\mathcal{F}:\mathbb{R}^{n \times n} \rightarrow \mathbb{R}^{n \times n}$ as follows. Given a matrix $P\in\mathbb{R}^{n\times n}$, consider the following partition
\begin{equation} 
\label{eqn:ppart}
 P=\m{E & G\\H & F}, \quad \mbox{ where } E\in\mathbb{R}^{n_+\times n_+} \mbox { and } F\in\mathbb{R}^{n_-\times n_-},
\end{equation}
and define
\begin{subequations} 
\label{eqn:SDA}
\begin{align}
\widehat{E} & := E(I - G H)^{-1}E, \label{eqn:SDA_Ek} \\
\widehat{F} & := F(I - HG)^{-1}F, \label{eqn:SDA_Fk} \\
\widehat{G} & := G + E(I - GH)^{-1}GF,  \label{eqn:SDA_Gk} \\ 
\widehat{H} & := H + F(I - HG)^{-1}HE.  \label{eqn:SDA_Hk} 
\end{align}
\end{subequations}
Then, 
\begin{align} 
	\label{eqn:doubling}
\mathcal{F}(P) :=  \m{\widehat{E} & \widehat{G}\\\widehat{H} & \widehat{F}}, \quad \mbox{ where } \widehat{E}\in\mathbb{R}^{n_+\times n_+} \mbox{ and } \widehat{F}\in\mathbb{R}^{n_-\times n_-}.
\end{align} 
The map $\mathcal{F}$ is 
known as the \emph{doubling} map, which is well-defined whenever the matrices $I-GH$ and $I-HG$ are nonsingular. 
Morever, define 
\begin{align} 
\label{eqn:alphabeta}
\alpha_{\mathrm{opt}}:=\min_{i\in\mathcal{S}_-} \abs*{\frac{C_{ii}}{T_{ii}}}, \quad  \beta_{\mathrm{opt}}:=\min_{i\in\mathcal{S}_+} \abs*{\frac{C_{ii}}{T_{ii}}}, 
\end{align}
and choose two nonnegative reals $\alpha \leq \alpha_{\mathrm{opt}}$ and $\beta \leq \beta_{\mathrm{opt}}$, not both being zero, and let 
\begin{align}
 \label{eqn:P0}
P_0:=Q^{-1}R 
\end{align} 
where 
\begin{align}
Q & := \m{C_+-\alpha T_{++} & -\beta T_{+-}\\ -\alpha T_{-+} & \abs{C_-}-\beta T_{--}}, \; 
R :=\m{C_++\beta T_{++} & \alpha T_{+-}\\ \beta T_{-+} & \abs{C_-}+\alpha T_{--}}\end{align}
Note that the conditions $\alpha\leq\alpha_{\mathrm{opt}}$ and $\beta\leq\beta_{\mathrm{opt}}$ are equivalent to requiring $R_{ii}\geq 0$ for each $i$.

Applying the doubling map to $P_0$ iteratively defines a sequence as follows
\begin{align}
 \label{eqn:Pk}
 P_k:=\mathcal{F}^k(P_0), \quad P_k:=\m{E_k & G_k\\ H_k & F_k}, \quad k \geq 0.
\end{align}
where $\mathcal{F}^k$ denotes the composition of $\mathcal{F}$ with itself $k$ times. The following theorem describes the limiting behavior of the four blocks of $P_k$. Here and in the following, inequalities among matrices are intended in the componentwise sense.
\begin{theo} 
\label{thm:sdaprops}
Consider the iteration \eqref{eqn:Pk} applied with initial values \eqref{eqn:P0}, where $(T,C)$ are the irreducible generator and the rate matrix of a positive recurrent fluid queue $\{X(t), \varphi(t)\}$. For each $k \geq 0$, the matrices $I-G_kH_k$ and $I-H_kG_k$ are nonsingular, and the matrices $E_k, F_k, G_k$, and $H_k$ satisfy the following properties: 
\begin{enumerate}
\item[(i)] $E_k > 0$ and 
$
\lim\limits_{k \rightarrow \infty} E_k = 0,
$
\item[(ii)] $F_k > 0$ and 
$
\lim\limits_{k \rightarrow \infty} F_k =: F_{\infty}, \mbox{ for some matrix } F_{\infty}.
$
\item[(iii)] $0 < G_k < G_{k+1}$ and $\lim\limits_{k\rightarrow \infty} G_k = \Psi$, where $\Psi$ is the minimal nonnegative solution of the NARE~\eqref{eqn:Psi}.
\item[(iv)] $0 < H_k < H_{k+1}$ and $\lim\limits_{k\rightarrow \infty} H_k = \widehat{\Psi}$, where $\widehat{\Psi}$ is the minimal nonnegative solution to the NARE
\begin{align} 
	\label{eqn:Psihat} 
|C_{-}|^{-1}T_{-+} + \widehat{\Psi}C_{+}^{-1}T_{++} + |C_{-}|^{-1}T_{--}\widehat{\Psi} + \widehat{\Psi} C_{+}^{-1}T_{+-} \widehat{\Psi} = 0.
\end{align} 
\item[(v)] The convergence rate in the above limits is quadratic, that is,
\begin{align}
 \label{quadconv}
 \abs{P_k-\lim_{k\to\infty} P_k} \leq K_0 \delta^{2^k}\bs{1}\bs{1}^{\top},
\end{align}
where $K_0>0$ is a suitable real number, and $\delta = \frac{1-\beta\lambda}{1+\alpha\lambda}$, $\lambda>0$ being the Perron eigenvalue of the $M$-matrix $-K$.
\end{enumerate} 
\end{theo} 
\begin{proof}
We focus on the case $\alpha>0$, $\beta>0$; when one among $\alpha$ and $\beta$ is $0$, the theorem follows from~\cite[Theorem~8]{bmp}. Most of the results are proved in~\cite[Theorem~3.3]{wwl}; the only missing part is the convergence of $E_k$ and $F_k$. This property does not indeed hold in general in the setting of~\cite{wwl}, but it does here thanks to the additional condition that the M-matrix $-\abs{C^{-1}}T$ associated to the Riccati equation is singular; hence we need to modify slightly their proof. Let us take $K$ as in~\eqref{eqn:K} and $\widehat{K}:=\abs{C_-^{-1}}T_{--}+\widehat{\Psi} C_+^{-1}T_{+-}$. It is known that they are irreducible $-M$-matrices, the former nonsingular and the latter singular in the positive recurrent case \cite{glr11b}. We have $I+\beta K \geq 0$ and $(I-\alpha K)^{-1}>0$, hence $U:=(I+\beta K)(I-\alpha K)^{-1}>0$. Similarly, $\widehat{U}:=(I+\alpha \widehat{K})(I-\beta \widehat{K})^{-1}>0$; in addition, since $\widehat{K}$ is singular, it follows that $\widehat{U}$ has Perron eigenvalue $\rho(\widehat{U})=1$. Applying \cite[Theorem~2.3(b)]{wwl} to $K$ and $\widehat{K}$, we get that $\rho(U)\rho(\widehat{U})<1$, hence $\rho(U)<1$. As a consequence, $\lim_{k\to\infty} U^{2^k}=0$, and $\lim_{k\to\infty} \widehat{U}^{2^k}$ is a rank-1 matrix with spectral radius $1$.

In our notation, \cite[Equations~3.15a and~3.15d]{wwl} read
\[
 E_k = (I-G_k\widehat{\Psi})U^{2^k}, \quad F_k = (I-H_k \Psi) \widehat{U}^{2^k},
\]
hence $\lim E_k = 0$ and $\lim F_k$ is finite.
\end{proof}
The iteration described above is called ADDA~\cite{wwl}; older variants appeared in literature are SDA-ss~\cite{bmp}, which corresponds to choosing $\alpha=0$, and SDA~\cite{glx05}, which corresponds to choosing $\alpha=\beta=\min(\alpha_{\mathrm{opt}},\beta_{\mathrm{opt}})$. In general, we call them \emph{doubling algorithms}. Clearly, Theorem~\ref{thm:sdaprops} also holds for both SDA-ss and SDA, but the values of $K_0$ and $\delta$ for the three algorithms are different in general. 
\begin{rem}
Here, there are two significant changes in the presentation of these algorithms with respect to the papers in which they originally appeared. First, the parameters $\alpha$ and $\beta$ here are the inverses of the parameters $\alpha'$ and $\beta'$ in the original papers. That is, $\alpha := 1/\alpha'$ and $\beta := 1/\alpha'$. This makes more apparent the relation with SDA-ss, since we can set $\alpha = 0$. 
Second, Equation \eqref{eqn:P0}, which can be derived using~\cite[Theorem~5.1]{pr11}, is a simpler formula for the initial values than those given in the original papers.
\end{rem}

\section{A cancellation-free version of doubling}
	\label{sec:cfSDA}
	
\subsection{M-matrices and the GTH-like algorithm}

We begin with some definitions. Let us define the \emph{off-diagonal} of an $m \times m$ matrix $A=(A_{i,j})$ as the vector $\offdiag(A)\in\mathbb{R}^{m^2-m}$ which contains all the non-diagonal entries of $A$ ordered one column after the other, that is, 
\begin{align*} 
 \offdiag(A) = (A_{2,1},  A_{3,1}, \ldots, A_{m,1}, A_{2,1}, A_{2,3}, \ldots, A_{2,m}, \ldots , A_{m,m-1})^\top.
\end{align*} 
(In fact, the actual order of these elements does not matter in our paper.) A square matrix $A$ is called a \emph{$Z$-matrix} if $\offdiag(A)\leq \bs{0}$, and an $m \times m$ $Z$-matrix $A$ is called an \emph{$M$-matrix} if there exists a vector $\bs{v}\in\mathbb{R}^m$ such that $\bs{v} > \bs{0}$ and $A \bs{v} \geq \bs{0}$. It is well-known \cite{bp94} that a nonsingular $Z$-matrix $A$ is an $M$-matrix if and only if $A^{-1}\geq 0$, and that a $Z$-matrix $A$ is an $M$-matrix if and only if all its eigenvalues have nonnegative real parts. In particular, the transpose of an $M$-matrix is also an $M$-matrix.

A triple $(\bs{u}, \bs{v}, \bs{w})\in \mathbb{R}^{m^2-m} \times \mathbb{R}^m \times \mathbb{R}^m$ with $\bs{u}\leq \boldsymbol{0}$, $\bs{v}>\boldsymbol{0}$, $\bs{w}\geq \boldsymbol{0}$ is called a \emph{right triplet representation} for the $Z$-matrix $A$ if $\offdiag(A)=\bs{u}$ and $A \bs{v} = \bs{w}$. In a sense, providing a triplet representation serves as a ``certificate'' that $A$ is an $M$-matrix, because it corresponds to verifying our definition explicitly. It is remarkable that the same certificate can be used to perform several matrix-algebraic operations with high accuracy. Indeed, we recall here the GTH-like algorithm, introduced in \cite{AlfXY02}, which is an algorithm for solving linear systems $A \bs{x} = \bs{b}$, with $A$ being a nonsingular $M$-matrix and $\bs{b} \geq \bs{0}$, using a triplet representation of $A$ and $\bs{b}$ as the only input. The diagonal of $A$ is never given explicitly, but it is uniquely determined by the triplet representation thanks to the relation $A_{ii} = (\bs{w}_i - \sum_{j\neq i} A_{ij}\bs{v}_j)/\bs{v}_i$.
We present the algorithm as Algorithm~\ref{algo:gth} below.
\begin{algorithm} \label{algo:gth}
 \caption{GTH-like algorithm for solving $A\bs{x} = \bs{b}$, where $A$ is a nonsingular $M$-matrix}
 \KwIn{$\bs{b} \in\mathbb{R}^{m}$ and a triplet representation ($\offdiag(A), \bs{v}, \bs{w})$, where\\
 \hspace*{0.5cm} $\offdiag(A)\in\mathbb{R}^{m^2-m}$,$\bs{v} = (v_1, \ldots, v_m) , \bs{w} = (w_1, \ldots, w_m),$ \\
 \hspace*{0.5cm} satisfying $\offdiag(A)\leq \bs{0}$, $\bs{v} > \bs{0}$, $\bs{w}, \bs{b} \geq \bs{0}.$}
 \KwOut{$\bs{x} =A^{-1}\bs{b}$}
 $L \leftarrow I_m$\;
 $U \leftarrow 0_{m\times m}$; $\offdiag(U)\leftarrow \offdiag(A)$\; \nllabel{U1}
 \For{$k=1 \text{ to } m$}{
   $U_{k,k} \leftarrow ({w_k} -U_{k,k+1:m}{v_{k+1:m}})/{v_k}$\; \nllabel{extra1}
   $L_{k+1:m,k} \leftarrow U_{k+1:m,k} / U_{k,k}$\;
   $w_{k+1:m} \leftarrow {w_{k+1:m}}-L_{k+1:m,k}{w_k}$\; \nllabel{extra2}
   $U_{k+1:m,k} \leftarrow 0$\;
   $\offdiag(U_{k+1:m,k+1:m}) \leftarrow \offdiag(U_{k+1:m,k+1:m} - L_{k+1:m,k}U_{k,k+1:m})$\; \nllabel{U2}
 }
 $\bs{y} \leftarrow L^{-1}\bs{b}$ (computed by forward-substitution)\;
 $\bs{x} \leftarrow U^{-1}\bs{y}$ (computed by back-substitution).
\end{algorithm}

The following result demonstrates the high level of stability of Algorithm~\ref{algo:gth} when implemented on a computer. Here and in the following, we use the common floating point arithmetic model in which $\operatorname{fl}(x \mathop{op} y) = (x \mathop{op} y)(1+\varepsilon)$, $\abs{\varepsilon}\leq {\mp}$, for $\mathop{op}\in\{+,-,\times, \div \}$ and machine precision $\mp$ \cite{highamstab}.
\begin{theo}[\cite{AlfXY02}] \label{th:gthstab}
	\label{theo:fourone}
Suppose that Algorithm~\ref{algo:gth} is carried out in floating-point arithmetic with machine precision ${\mp}$ to solve the linear system $A\bs{x}=\bs{b}$, where $\offdiag(A),\bs{v}, \bs{w},$ and $\bs{b}$ are floating point numbers, and $(\offdiag(A),\bs{v}, \bs{w})$ is a triplet representation for $A$. Then, the computed solution $\tilde{\bs{x}}$ satisfies the following inequality 
 \begin{equation} 
 		\label{gthbound}
  \abs{\tilde{\bs{x}} - \bs{x}} \leq (\psi(m){\mp}+O({\mp}^2)) \bs{x},
 \end{equation}
where $\psi(m):=\frac{2}{3}(2m+5)(m+2)(m+3).$ 
\end{theo}
Notice in~\eqref{gthbound} the surprising absence of the condition number of $A$, which one would expect from classical error analysis \cite[Chapters~7 and~9]{highamstab}. This is possible only because the input values of the algorithm are not the entries of $A$ but a different object, its triplet representation. In addition, we point out that the term $\psi(m)$ is usually a pessimistic estimate: in practice, the growth factor is typically $O(m)$ rather than $O(m^3)$~\cite{AlfXY02}.

The GTH-like algorithm works by computing an LU factorization of the matrix $A$; it differs only slightly from a standard LU-based linear system solver \cite{gvl96}. Indeed, the only difference here are Lines~\ref{extra1} and~\ref{extra2}, which compute each diagonal entry of $U$ with an equivalent subtraction-free expression instead of the usual update rule (that is, performing Lines~\ref{U1} and~\ref{U2} on the whole matrix $U$ rather than on the off-diagonal entries only).
It is easy to use the same LU factorization to solve linear systems of the form $A^\top \bs{x} = \bs{b}$ or to compute the (left or right) one-dimensional kernel of a singular $M$-matrix, by simple variants of Algorithm~\ref{algo:gth}.

In a similar fashion, we say that $(\offdiag(A), \bs{v}^\top, \bs{w}^\top)$ is a \emph{left triplet representation} for $A$ if $\bs{v}^\top A = \bs{w}^\top$, with $\bs{v}^\top > \bs{0}$ and $\bs{w}^\top \geq \bs{0}$ (or, equivalently, if $(\offdiag(A^\top), \bs{v}, \bs{w})$ is a right triplet representation for $A^\top$). It is clear that we can adapt the GTH-like algorithm to work with left triplet representations.

\subsection{Triplet representations in doubling}

Recently, an implementation of SDA taking advantage of the GTH-like algorithm was proposed in \cite{xxl12}, but triplet representations for the matrices to invert in the algorithm, $I-H_kG_k$ and $I-G_kH_k$, were not provided. Rather, the authors advised to compute them from the entries of the matrices using an iterative algorithm. However, when a matrix is ill-conditioned, computing its triplet representation starting from its entries is an ill-conditioned problem, hence most of the benefits of using the GTH-like algorithm are lost in this way.

We show in the following theorem that triplet representations for all the matrices to be inverted in the algorithm are available explicitly with subtraction-free formulas.
\begin{theo} \hspace*{0.2cm} 

\begin{enumerate} 
\item A triplet representation for $Q$, defined in~\eqref{eqn:P0}, is 
\begin{align}
(\offdiag(Q),\ones,\abs{C}\ones). \label{eqn:tripCT}
\end{align} 
\item A triplet representation for $I-G_k H_k$ is 
\begin{align} 
(\offdiag(-G_k H_k),\mathbf{1},G_kF_k\mathbf{1}+E_k \mathbf{1}). \label{eqn:tripGH}
\end{align} 
\item A triplet representation for $I-H_k G_k$ is
\begin{align} 
(\offdiag(-H_k G_k),\mathbf{1},H_kE_k\mathbf{1}+F_k \mathbf{1}). \label{eqn:tripHG} 
\end{align}  
\end{enumerate} 
\end{theo}
\begin{proof}
The first representation follows easily from $T\ones = \bs{0}$. For the second representation, since the matrix $P_k$ is stochastic for $k \geq 0$ (see Corollary~\ref{Pstoc} in the Appendix),
\begin{align}
\label{Pke}
 \m{I & -G_k\\ -H_k & I}\ones=\m{E_k & 0\\ 0 & F_k}\ones.
\end{align}
Pre-multiplying both sides by $\m{I & G_k}$ gives
\begin{align*} 
\m{I-G_kH_k & \bs{0}}\ones=\m{E_k & G_kF_k}\ones,
\end{align*}
which implies that $(I-G_k H_k)\ones = G_kF_k\mathbf{1}+E_k \mathbf{1}$. Similarly, pre-multiplying both sides of \eqref{Pke} by $\m{H_k & I}$ we get the equality needed for the third triplet representation~\eqref{eqn:tripHG}.
\end{proof}
Using these triplet representations, we can perform in a subtraction-free fashion with the GTH-like algorithm all the inversions required for computing the matrix $\Psi$ starting from $T$ and $C$. However, going through the algorithm outlined in Section~\ref{sec:SDA}, there is still a place left where we need subtractions: when computing $R$ using \eqref{eqn:P0}, we need to evaluate the diagonals of the matrices $C_+ + \beta T_{++}$ and $\abs{C_-}+\alpha T_{--}$. 

Recall that the parameters $\alpha$ and $\beta$ are chosen so that these diagonal entries are positive, and the optimal values $\alpha_{\mathrm{opt}}$ and $\beta_{\mathrm{opt}}$ are the smallest ones that achieve this goal. The same issue is discussed in \cite[Section~4.3]{xxl12}; a simple fix is choosing $\alpha =\eta\alpha_{\mathrm{opt}}$ and $\beta =\eta\beta_{\mathrm{opt}}$ for some $\eta<1$ not too close to $1$ (for instance, $\eta=0.9$ or $\eta=0.5$). In this way, the algorithm still contains $n$ subtractions but they are numerically safe, since we guarantee that the two terms are not too close to each other.

To summarize, our proposed algorithm is presented as Algorithm~\ref{algo:cwsda}.
\begin{algorithm}[ht] \label{algo:cwsda}
\caption{Componentwise accurate SDA}
\KwIn{$\offdiag(T)$, $C$, $\varepsilon$, where \\
\hspace*{0.1cm} $T\in\mathbb{R}^{n\times n}$ is the transition matrix of a continuous-time Markov chain,  \\
\hspace*{0.1cm} $C\in\mathbb{R}^{n\times n}$ is a diagonal rate matrix with positive rates in the first $n_{+}$ entries and negative ones in the remaining $n_{-} = (n - n_{+})$ entries, \\
\hspace*{0.1cm} $\varepsilon$ is the convergence tolerance.}
\KwOut{The minimal nonnegative solutions $\Psi\in\mathbb{R}^{n_+\times n_-}$ of~\eqref{eqn:Psi} and $\widehat{\Psi}\in\mathbb{R}^{n_-\times n_+}$ of~\eqref{eqn:Psihat}}
Set $\alpha \leftarrow \eta \alpha_{\mathrm{opt}}$, $\beta \leftarrow \eta \beta_{\mathrm{opt}}$ using, for instance, $\eta=0.5$\;
Compute $T_{ii}=-\sum_{i\neq j}T_{ij}$\;
Compute initial values $E_0,F_0,G_0,H_0$ according to \eqref{eqn:P0}, solving the required linear systems with Algorithm~\ref{algo:gth} and triplet representation~\eqref{eqn:tripCT}\;
$k \leftarrow 0$\;
\Repeat{$\abs{G_{k+1}-G_{k}}\leq \varepsilon G_{k+1}$}{
 Compute $E_{k+1},F_{k+1},G_{k+1}, H_{k+1}$ according to \eqref{eqn:SDA}, using Algorithm~\ref{algo:gth} for inversions with triplet representations \eqref{eqn:tripGH} and~\eqref{eqn:tripHG}\;
 $k\leftarrow k+1$\;
}
$\Psi \leftarrow G_{k}$; $\widehat{\Psi} \leftarrow H_{k}$. 
\end{algorithm}
\begin{rem}
If one wishes to avoid subtractions completely, the following strategy can be used. Choose 
\begin{align*} 
\alpha :=\left(\sum_{k\in\mathcal{S}_-} \frac{\abs{T_{kk}}}{\abs{c_k}}\right)^{-1}=\left(\sum_{k\in\mathcal{S}_-,j\in\mathcal{S},j\neq k} \frac{T_{kj}}{\abs{c_k}}\right)^{-1}.
\end{align*}
Now, for each $i\in\mathcal{S}_-$ we can write
\begin{align} 
\abs{c_i}+\alpha T_{ii} & =\alpha\abs{c_i}\left(\alpha^{-1}-\frac{\abs{T_{ii}}}{\abs{c_i}}\right) \nonumber 
= \alpha\abs{c_i}\left(\sum_{k\in\mathcal{S}_-,j\in\mathcal{S},k\not\in \{i,j\}} \frac{T_{kj}}{\abs{c_k}} \right), \nonumber
\end{align} 
and the last expression contains only sums, products and inverses of nonnegative numbers. A similar expression holds for $\beta$, replacing $\mathcal{S}_-$ with $\mathcal{S}_+$. This choice of $\alpha$ satisfies $\frac{\alpha_{\mathrm{opt}}}{n_-}\leq \alpha \leq \alpha_{\mathrm{opt}}$ (and correspondingly $\frac{\beta_{\mathrm{opt}}}{n_+}\leq \beta \leq \beta_{\mathrm{opt}}$), so it is also not excessively far from the optimal value. \hfill $\Box$
\end{rem}
How does choosing values of $\alpha$ and $\beta$ lower than $\alpha_{\mathrm{opt}}$ and $\beta_{\mathrm{opt}}$ affect the convergence speed? The function $\delta=\delta(\alpha,\beta)$ is decreasing in $\alpha$ and $\beta$ \cite[Theorem~2.3]{wwl}, so this additional safeguard makes the algorithm slower. Nevertheless, thanks to the quadratic nature of the convergence, the slowdown can typically be measured by a small number of additional iterations. Indeed, consider the most troublesome case for convergence, the one in which $\lambda \approx 0$. In this case, we have $\delta(\alpha,\beta)^{2^k} \approx 1 - 2^k(\alpha +\beta)\lambda$; hence, multiplying $\alpha$ and $\beta$ by a factor $\frac{1}{m}$ is roughly equivalent to performing $\log_2 m$ more iterations in SDA, since $2^k(\alpha +\beta) = 2^{k+\log_2 m} \left(\frac{\alpha}{m}+\frac{\beta}{m}\right)$.

\subsection{Numerical accuracy}

Relying on the results on the stability of the GTH-like algorithm, we can obtain explicit error bounds for the accuracy of the SDA algorithm. These bounds are stated in the following theorem. 

\begin{thm} 
	\label{Gtrick-above}
Let $\tilde{G}_k$ be the approximation of the matrix $G_k$ computed by Algorithm~\ref{algo:cwsda}, performed in floating-point arithmetic with machine precision ${\mp}$ using $\alpha\leq 2\alpha_\mathrm{opt}$ and $\beta\leq 2\beta_\mathrm{opt}$. Then,
\begin{align}
	\label{eqn:theo44}
  \abs{\tilde{G}_k-G_k} \leq \left(\left(k+K_1+\frac{2(K_1+K_2)n}{1-\delta}\tilde{K}_0\right)\mp+O(\mp^2)\right)\Psi,
\end{align}
where $\tilde{K}_0>0$ is a constant such that $\Psi-G_k \leq \tilde{K}_0 \delta^{2^k}\Psi$, 
\begin{align*} 
K_1 & :=\psi(n)+5n+2, \\
K_2 & :=\psi(n_{\max}) + 4n_{\max}^2+2n_{\max}-1, \quad \mbox{ with } n_{\max}=\max(n_+,n_-).
\end{align*} 
\end{thm}
We recall that $\psi(n)$, defined in Theorem~\ref{th:gthstab}, is a degree-3 polynomial. The proof of Theorem~\ref{Gtrick-above} is rather technical and lengthy; therefore, we defer it to the Appendix.


\subsection{Other triplet representations}
We now focus on the computation of the stationary densities as described in Section~\ref{sec:statdis}, relying on the solutions computed with Algorithm~2.
Using the Riccati solutions $\Psi$ and $\widehat{\Psi}$, one can obtain the four matrices 
\begin{align*}
W &= T_{--} + T_{-+}\Psi, & K &=C_+^{-1}T_{++}+\Psi\abs{C_-^{-1}} T_{-+},\\
 \widehat{W} &= T_{++} + T_{+-}\widehat{\Psi}, & \widehat{K}&=\abs{C_-^{-1}}T_{--}+\widehat{\Psi} C_+^{-1}T_{+-},
\end{align*}
which appear also with a physical interpretation in the probabilistic setting of this problem \cite[Section~4]{glr11b}. The matrix $K$ was already defined in this paper in~\eqref{eqn:K}, while $W$ appeared in \eqref{eqn:invprob1} and $\widehat{K}$ in the proof of Theorem~\ref{thm:sdaprops}.
They are all $-M$-matrices, so the computation of their off-diagonal entries can be done directly without subtractions. In addition, triplet representations for their negatives can be computed explicitly.
\begin{lem} \label{lem:KW} \hspace{0.0cm}  
\begin{enumerate} 
\item $(\offdiag(-W),\mathbf{1}, \bs{0})$ and $(\offdiag(-\widehat{W}),\mathbf{1},T_{-+} F_\infty \mathbf{1})$ are right triplet representations for $-W$ and $-\widehat{W}$, respectively. 

\item $(\offdiag(-K), \bs{\xi_+}C_+, \bs{\xi_-}\abs{C_-} F_\infty \abs{C_-^{-1}} T_{-+})$ and $(\offdiag(-\widehat{K}), \bs{\xi_-}\abs{C_-}, \bs{0})$ are left 
triplet representations for $-K$ and $-\widehat{K}$, respectively. 
\end{enumerate} 
\end{lem}
\begin{proof}
Analogously to \eqref{Pke}, by Corollary~\ref{Pstoc}
\begin{align*}
 \m{\bs{\xi_+}C_+ & \bs{\xi_-}\abs{C_-}} \m{I & -G_k\\-H_k & I} = \m{\bs{\xi_+}C_+ & \bs{\xi_-}\abs{C_-}}\m{E_k & 0 \\ 0 & F_k}.
\end{align*}
We take the limit as $k \rightarrow \infty$ of the second column of both sides to obtain 
%
\begin{align*} 
\bs{\xi_-}\abs{C_-} - \bs{\xi_+}C_+ \Psi = \bs{\xi_-}\abs{C_-} F_\infty.
\end{align*}
Since $\bs{\xi} T = \bs{0}$, we have $\bs{\xi_+}T_{++}+\bs{\xi_-}T_{-+} = \bs{0}$ and therefore
\begin{align*}
-\bs{\xi_+}C_+ K & = (\bs{\xi_-}\abs{C_-} - \bs{\xi_+}C_+\Psi)\abs{C_-^{-1}}T_{-+}
                    = \bs{\xi_-}\abs{C_-} F_\infty \abs{C_-^{-1}}T_{-+}.
\end{align*}
The other representations are proved in a similar way, using \eqref{Pke} and the fact that $E_k\to 0$ as $k \rightarrow \infty$. 
\end{proof}



\subsection{Matrix exponentials, scaling and squaring, and SDA} \label{subsec:exp}
The problem of computing $e^{Kx}$ in \eqref{realinv} with high componentwise accuracy is discussed in~\cite{sgx12,xy08,xy13}, where different algorithms are presented for computing $\exp(A)$, $A$ being a $-Z$-matrix. For instance, one of the simplest algorithms consists in
\begin{enumerate} 
\item[(i)] choosing $\widehat{A} \geq 0$ and $z\geq 0$ such $A=\widehat{A}-zI$, and \item[(ii)] computing $\ue^A$ as
\begin{align} 
	\label{eqn:scalingandsquaring}
 \ue^A \approx \ue^{-z} \left(\mathcal{T}_m(2^{-s}\widehat{A}) \right)^{2^s},
\end{align}
\end{enumerate}
where $\mathcal{T}_m$ is the Taylor series of the exponential, truncated after its $m$th term, and $s$ is a suitable positive integer whose magnitude depends on the norm of~$A$.
%

The papers~\cite{xy08,xy13} discuss the computation of essentially nonnegative matrices to high componentwise accuracy, with the componentwise condition number obtained in~\cite{xy08} and the error analysis for~\eqref{eqn:scalingandsquaring} presented in~\cite{xy13}.

Algorithms of the form \eqref{eqn:scalingandsquaring}, where $\mathcal{T}_m$ may be a rational approximation rather than a Taylor expansion, are known as \emph{scaling and squaring} \cite{highamfun}. Matlab's standard algorithm \texttt{expm}, for instance, uses an algorithm of this family \cite{higham-scalsquar}, and we found that it often delivers componentwise accurate results on the matrix on which we are using it, although it is not explicitly designed for this purpose\footnote{For an analysis of the componentwise accuracy of scaling and squaring algorithms, see~\cite{acf}.}.
It is interesting to see that doubling is, indirectly, a scaling-and-squaring algorithm. Indeed, at each step the factorization
\begin{align} 
\label{eqn:scalsquarsda}
\left[\begin{array}{rr} 
I & -G_k\\ 
 0 & F_k
 \end{array}\right]^{-1} 
\left[\begin{array}{rr}
E_k & 0\\ -H_k & I
\end{array} \right]
 & = \left((I_n-\alpha C^{-1}T)^{-1}(I_n+\beta C^{-1}T)\right)^{2^k}
\end{align}
holds. The right-hand side is the same expression that would be obtained by applying a scaling and squaring method for computing $\exp(2^{k}(\alpha+\beta) C^{-1}T)$, with the inner function
\[
 \mathcal{T}(x) = \frac{1+\frac{\beta}{\alpha+\beta}x}{1-\frac{\alpha}{\alpha+\beta}x},
\]
which is indeed a rational first-order approximation of $e^x$.

The formula \eqref{eqn:scalsquarsda} is known in the SDA literature, but, as far as we know, the observation that this corresponds to a scaling and squaring method, albeit simple, is novel. \hfill $\Box$

\subsection{Density computation}
Putting everything together, we can evaluate the stationary density $\bs{f}(x)$ and the stationary probability mass at level zero of a Markov-modulated fluid queue $\{X(t),\varphi(t)\}$ using Algorithm~\ref{algo:density}. Note that all the computations are subtraction-free; in particular, for Line~\ref{pminus}, recall that $K^{-1}\leq 0$ while $V\geq 0$ and $T_{-+}\geq 0$.
\begin{algorithm}[ht]
\caption{Stationary density and probability mass of a fluid queue} \label{algo:density}
\KwIn{The irreducible generator $T\in\mathbb{R}^{n\times n}$, diagonal rate matrix $C\in\mathbb{R}^{n\times n}$, $x>0$.}
\KwOut{The stationary density $\bs{f}(x)$ and the stationary probability mass $\bs{p}_{-}$ at level zero.}
Compute $\Psi$ using Algorithm~\ref{algo:cwsda}\;
Compute a vector $\bs{q}$ such that $\bs{q}W=\bs{0}$ using (a variant of) the GTH-like algorithm and the triplet representation in Lemma~\ref{lem:KW}\;
Compute $\bs{p_-}=\beta^{-1}\bs{q}$, where $\beta=\bs{q}\left(\mathbf{1}-T_{-+}K^{-1}V\mathbf{1}\right)$\; \nllabel{pminus}
Compute $K$ and $V$ according to \eqref{eqn:K} and \eqref{eqn:V}\;
Compute $\bs{f}(x)$ according to \eqref{realinv}, evaluating the matrix exponential as described in Section~\eqref{subsec:exp} or in \cite{sgx12}, or with Matlab's \texttt{expm}\;
\end{algorithm}

\section{Numerical experiments}  
	\label{sec:ne}

We present two examples to demonstrate that updating correctly the triplet representations is crucial for attaining componentwise accuracy in the algorithm. For simplicity, we focus on the SDA and choose $\eta=0.5$; hence, $\alpha=\beta=0.5 \min(\alpha_\mathrm{opt},\beta_\mathrm{opt})$). \\

\noindent \textbf{Example 6.1 (weakly connected queue).} Consider a weakly recurrent fluid model $\{X(t), \varphi(t)\}$, where the phase process $\varphi(t)$ has the state space $\mathcal{S} = \{1, 2, 3, 4, 5, 6\}$ and the transition matrix 
\begin{align} 
T = \left[\begin{array}{rcrrcr}
  -4      &      0     &        0       &     0        &    0  &4 \\ 
            0 & -(15 + 10^{-8}) & 5&   5  & 5  & 10^{-8}  \\ 
            0 &  5& -15&   5 &  5     &       0 \\ 
            0 &  5 & 5  &-15  & 5         & 0 \\ 
            0  & 5 &  5  & 5 & -15       &     0 \\ 
   4 & 1       &     0      &      0    &        0 & -5 \end{array}\right],
\end{align} 
%
%
and $C = \diag(1, 1, 1, -1.001, -1.001,  -1.001)$. The transition graph of the phase process $\varphi(t)$ is depicted in Figure~\ref{fig:ex1}. 

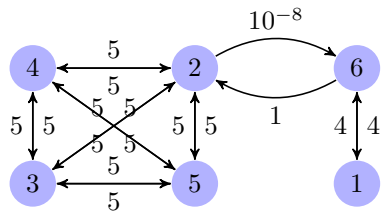
\begin{figure}[ht]
\centering
\begin{tikzpicture}[node distance=0.9cm]
 \node[states] (aa) {$4$};
 \node[states] (bb) [right=1.5 cm of aa] {$2$};
 \node[states] (cc) [below=of aa] {$3$};
 \node[states] (dd) [right=1.5cm of cc] {$5$};
 \node[states] (ee) [right=1.5cm of bb] {$6$};
 \node[states] (ff) [below=of ee] {$1$};
 
 \path (aa) edge[prob] node{5} (bb)
       (aa) edge[prob] node{5} (cc)
       (aa) edge[prob] node{5} (dd)
       (bb) edge[prob] node{5} (aa)
       (bb) edge[prob] node{5} (cc)
       (bb) edge[prob] node{5} (dd)
       (cc) edge[prob] node{5} (aa)
       (cc) edge[prob] node{5} (bb)
       (cc) edge[prob] node{5} (dd)
       (dd) edge[prob] node{5} (bb)
       (dd) edge[prob] node{5} (cc)
       (dd) edge[prob] node{5} (aa)
       (bb) edge[prob, bend left] node{$10^{-8}$} (ee)
       (ff) edge[prob] node{4} (ee)
       (ee) edge[prob] node{4} (ff)
       (ee) edge[prob, bend left] node{1} (bb)
	;
\end{tikzpicture}
\caption{The modulating Markov chain of a fluid queue, with two weakly connected sets of states, $\{4, 2, 3, 5\}$ and $\{1, 6\}$.} \label{fig:ex1}
\end{figure}
This is a fluid queue whose phase process is composed of two weakly connected sets of states: the states $\{1, 6\}$ are difficult to reach from the others due to the small transition rate $10^{-8}$. Hence, we expected the stationary density of the queue to have very small values in these two states. We computed the exact solution $\Psi$ with a high number of significant digits using variable precision arithmetic, and compared the normwise and componentwise errors defined as 
\begin{align*}
 e_{\mathrm{norm}} :=\frac{\norm{\tilde{\Psi}-\Psi}}{\norm{\Psi}}, \quad e_{\mathrm{cw}} :=\max \frac{\abs{\tilde{\Psi}_{ij}-\Psi_{ij}}}{\Psi_{ij}}.
\end{align*}
The resulting matrix
\begin{align*} 
 \Psi \approx \m{0.195 & 0.195 & 0.61\\ 0.5 & 0.5 & 2\cdot 10^{-9}\\ 0.5 & 0.5 & 1.7\cdot 10^{-9}}
\end{align*} 
has two entries, $\Psi_{2,3}$ and $\Psi_{3,3}$, that are much smaller in magnitude than the others. Probabilistically, these correspond to returning to the initial level in the hard-to-reach state $6$ after starting from states $2$ and $3$, respectively. We compare in Table~\ref{tbl:ex1} the results of three algorithms:

\begin{description}
 \item[GLX-SDA] the non-componentwise accurate ``standard'' version of SDA introduced in Guo \emph{et al.} \cite{glx05},
 \item[XXL-SDA] the variant suggested by Xue \emph{et al.} \cite{xxl12}, which employs triplet representations but recomputes them from the matrix entries at each step;
 \item[comp-SDA] the componentwise accurate version of SDA that we introduce in this paper (Algorithm~\ref{algo:cwsda}).
\end{description}
\npproductsign{\ensuremath{\cdot}}
\begin{table}[ht!]
\centering
\renewcommand{\arraystretch}{0.5}
\begin{tabular}{n{2}{1}n{2}{1}n{2}{1}n{2}{1}n{2}{1}n{2}{1}}
\toprule
\multicolumn{2}{c}{GLX-SDA} &
\multicolumn{2}{c}{XXL-SDA} &
\multicolumn{2}{c}{comp-SDA} \\
\midrule
 \multicolumn{1}{c}{$e_{\mathrm{norm}}$} & \multicolumn{1}{c}{$e_{\mathrm{cw}}$} &
 \multicolumn{1}{c}{$e_{\mathrm{norm}}$} & \multicolumn{1}{c}{$e_{\mathrm{cw}}$} &
 \multicolumn{1}{c}{$e_{\mathrm{norm}}$} & \multicolumn{1}{c}{$e_{\mathrm{cw}}$} \\
 \midrule
 1.3e-13 & 2.1e-09 & 1.0e-13 & 3.6e-13 & 2.0e-16 & 9.3e-16\\
 \bottomrule
\end{tabular}
\caption{Accuracy of various algorithms on the queue in Figure~\ref{fig:ex1}} \label{tbl:ex1}
\end{table}

The non-componentwise accurate variant GLX-SDA computes a normwise-accurate matrix~$\tilde{\Psi}$, but the two small entries are computed with a larger relative error:
\[
 \left[\frac{\tilde{\Psi}_{ij}-\Psi_{ij}}{\Psi_{ij}}\right]  \approx \m{\numprint{-4e-13} & \numprint{-4e-13} & \numprint{1e-16} \\ \numprint{-6e-14} & \numprint{-6e-14} & \numprint{-2e-9} \\ \numprint{-6e-14} & \numprint{-6e-14} & \numprint{-2e-9}}.
\]
The algorithm XXL-SDA reaches a smaller error on these tiny entries, but the overall error is of the same order of magnitude as the standard algorithm, and is far from reaching the full machine precision of 16 significant digits. 

As a second test encompassing the whole Algorithm~\ref{algo:density}, we computed the value of the probability density function $\bs{f}(x)$ with the three algorithms for several values of $x$ spanning different orders of magnitude. We used Matlab's \texttt{expm} function rather than the simple one described in Section~\ref{subsec:exp}, since we found that it yields slightly better results on this example, despite having no explicit componentwise accuracy guarantees. The normwise errors on the resulting vectors $\bs{f}(x)$ are reported in Figure~\ref{fig:stateerror}, while the componentwise ones in Figure~\ref{fig:stateerror2}. 
 \pgfplotsset{
/pgfplots/bar cycle list/.style={/pgfplots/cycle list={%
{red,fill=red!30!white,mark=none},%
{green,fill=green!30!white,mark=none},%
{blue,fill=blue!30!white,mark=none},%
}
},
}
 \pgfplotstableread{pdf2014.dat}{\results}
 \begin{figure}[h]
 \caption{Normwise error for the stationary density $\bs{f}(x)$ computed with $\Psi$ obtained by different algorithms.}\label{fig:stateerror}
 \begin{tikzpicture}
 \begin{axis}[height=0.4\textheight,width=\textwidth,xmode=log, ymode=log, ymax=1, ymin=1e-16,ylabel=normwise error on $\bs{f}(x)$,xlabel=$x$, legend style={at={(0.08,-0.17)},anchor=north west},legend columns=4,thick]
  \addplot[color=red,dashdotdotted] table[x=point,y=vanilla_error] {\results};\addlegendentry{GLX-SDA};
  \addplot[color=green,densely dashed] table[x=point,y=li2_error] {\results}; \addlegendentry{XXL-SDA};
  \addplot[color=blue,densely dashdotted] table[x=point,y=acc2_error] {\results}; \addlegendentry{comp-SDA}
  \addplot[color=violet,densely dotted] table[x=point,y=cond] {\results}; \addlegendentry{\texttt{cond(expm)}}
  \addplot[color=orange] table[x=point,y=density] {\results}; \addlegendentry{$\bs{f}(x)\bs{1}$}
 \end{axis}
\end{tikzpicture}
 \end{figure}
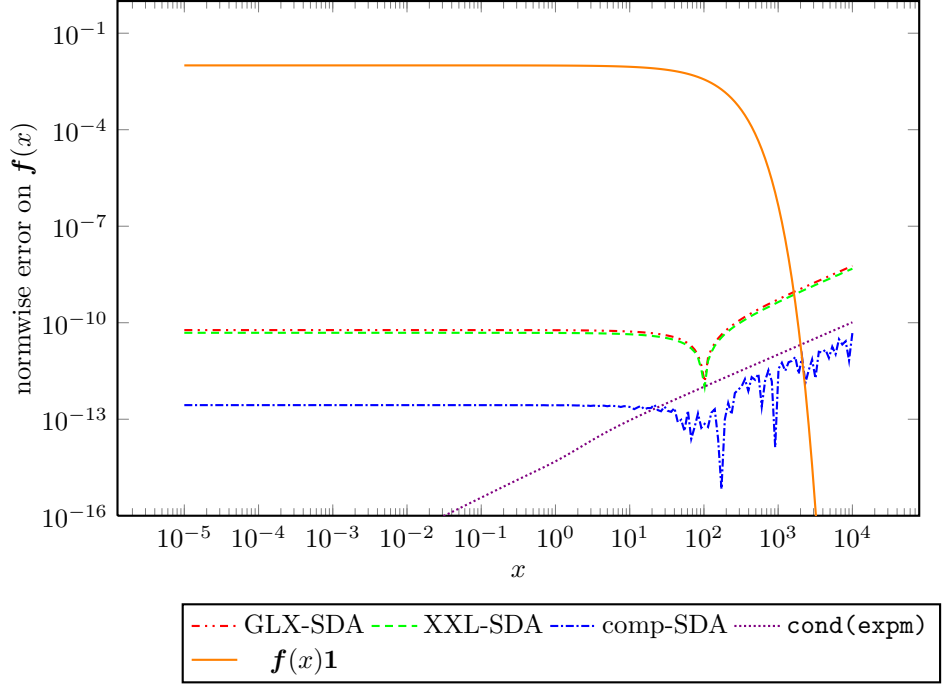
 
 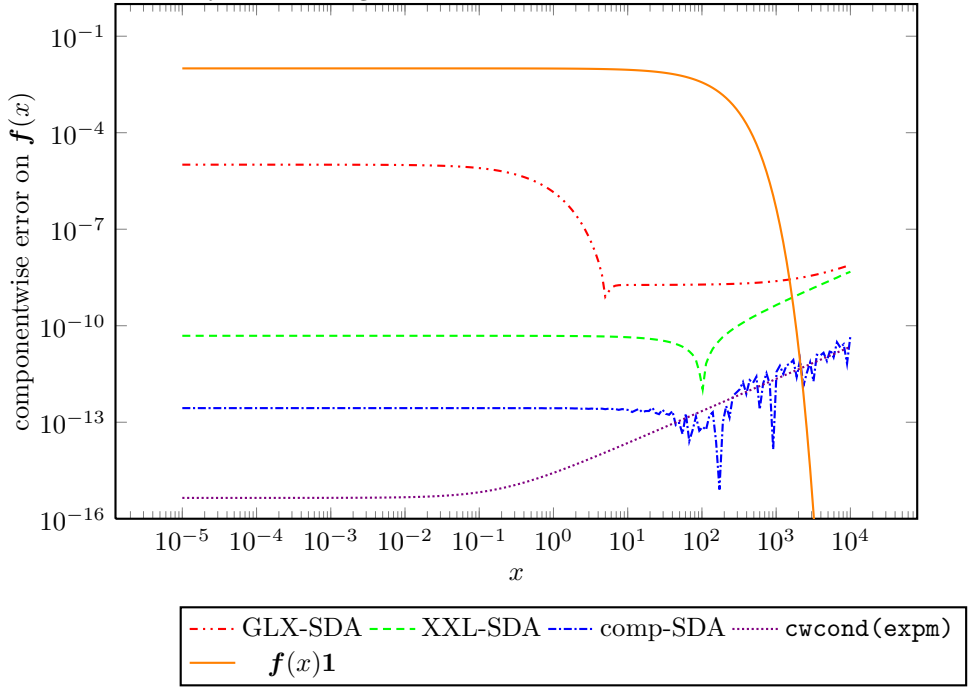
\begin{figure}
 \caption{Componentwise error for the stationary density $\bs{f}(x)$ computed with $\Psi$ obtained by different algorithms.}\label{fig:stateerror2}
 \begin{tikzpicture}
 \begin{axis}[height=0.4\textheight,width=\textwidth,xmode=log, ymode=log,
 ymax=1, ymin=1e-16,ylabel=componentwise error on $\bs{f}(x)$,xlabel=$x$, legend style={at={(0.08,-0.17)},anchor=north west},legend columns=4,thick]
  \addplot[color=red,dashdotdotted] table[x=point,y=cw_vanilla_error] {\results};\addlegendentry{GLX-SDA};
  \addplot[color=green,densely dashed] table[x=point,y=cw_li2_error] {\results}; \addlegendentry{XXL-SDA};
  \addplot[color=blue,densely dashdotted] table[x=point,y=cw_acc2_error] {\results}; \addlegendentry{comp-SDA}
  \addplot[color=violet,densely dotted] table[x=point,y=cw_cond] {\results}; \addlegendentry{\texttt{cwcond(expm)}}
  \addplot[color=orange] table[x=point,y=density] {\results}; \addlegendentry{$\bs{f}(x)\bs{1}$}
 \end{axis}
\end{tikzpicture}
 \end{figure}
 The plots contain also an estimate of the condition number of the matrix exponential, computed using the method in~\cite{higham_mft} for the normwise condition number and~\cite[Equation~2.1]{sgx12} for the componentwise condition number, and the aggregate magnitude of the density function $\bs{f}(x)\bs{1}$. As one can see, for higher values of the fluid level $x$, the result cannot be computed with full machine precision anymore, and this is due to the higher condition number of the matrix exponential for large arguments, which dominates the error. This happens roughly at the same order of magnitude as the decay of the density function.
 
 Notice also the unexpected dip of the error in correspondence of $x=10^{2}$ --- this seems to be a numerical artifact, possibly due to the computation of the matrix exponential; its position seems to vary when the parameter $-1.001$ in the matrix $C$ is altered, but we do not have a compelling probabilistic or numeric explanation for its existence.
 \hfill $\blacksquare$\\
~\\
\textbf{Example 6.2 (cascading process).} Consider a fluid queue $\{X(t), \varphi(t)\}$ whose transition graph for $\varphi(t)$ is depicted in Figure~\ref{fig:test}, with state space $\mathcal{S} = \{1, 2, \ldots, 8\}$, generator $T$ given by 
\begin{align*} 
T = \left[\begin{array}{rrrrrrrrr} 
          -1     &       0   &         0  &          0    &         0         &    0      &       0  &  1.00  \\
            0   & -1.01   &         0   &         0  &  0.01   &          0   &          0  &  1.00 \\ 
            0   &        0  &    -1.01  &    0    &       0  &  0.01      &     0  &  1.00 \\
            0    &       0   &       0     &  -1.01  &     0  &          0  &  0.01  &  1.00 \\
         0.01  &       0   &       0     &     0  & -1.01      &       0         &    0  &  1.00 \\ 
            0  &     0.01  &      0      &    0   &       0  &  -1.01        &    0  &  1.00 \\ 
            0   &         0  &   0.01    &    0  &           0    &         0  & -1.01 &   1.00 \\ 
            0   &         0   &        0  &  0.01 &           0         &    0 &            0  & -0.01 
\end{array}\right],
\end{align*} 
and $C = \diag(\kappa, 1, 1, 1, -1, -1, -1, -1) $ for a real parameter~$\kappa$. The state space $\mathcal{S}$ consists of a ``base'' state, labeled $8$, and a sequence of states each of which can be reached only by the previous one with a low transition rate; from each state, there is also a higher transition rate to the base state (see Figure~\ref{fig:test}).
\begin{figure}[h!]
\centering
\begin{tikzpicture}[node distance=0.9cm]
 \node[states] (aa) {8};
 \node[states] (bb) [right=of aa] {4};
 \node[states] (cc) [right=of bb] {7};
 \node[states] (dd) [right=of cc] {3};
 \node[states] (ee) [right=of dd] {6};
 \node[states] (ff) [right=of ee] {2};
 \node[states] (gg) [right=of ff] {5};
 \node[states] (hh) [right=of gg] {1};
 
  \path (aa) edge[prob] node {0.01} (bb)
	(bb) edge[prob] node {0.01} (cc)
	(cc) edge[prob] node {0.01} (dd)
	(dd) edge[prob] node {0.01} (ee)
	(ee) edge[prob] node {0.01} (ff)
	(ff) edge[prob] node {0.01} (gg)
	(gg) edge[prob] node {0.01} (hh)
	
	(bb) edge[prob,bend left=10,out=20,in=170] node [pos=0.1] {1} (aa)
	(cc) edge[prob,bend right=9,out=20,in=160] node [pos=0.1]{1} (aa)
	(dd) edge[prob,bend right=8,out=20,in=150] node [pos=0.1]{1} (aa)
	(ee) edge[prob,bend right=7,out=20,in=140] node [pos=0.1]{1} (aa)
	(ff) edge[prob,bend right=6,out=20,in=130] node [pos=0.1]{1} (aa)
	(gg) edge[prob,bend right=5,out=20,in=120] node [pos=0.1]{1} (aa)
	(hh) edge[prob,bend left=4,out=20,in=110] node [pos=0.1]{1} (aa)

	;
\end{tikzpicture}
\vspace*{-1.6cm} 
\caption{A fluid queue model for a cascading process} \label{fig:test}
\end{figure}
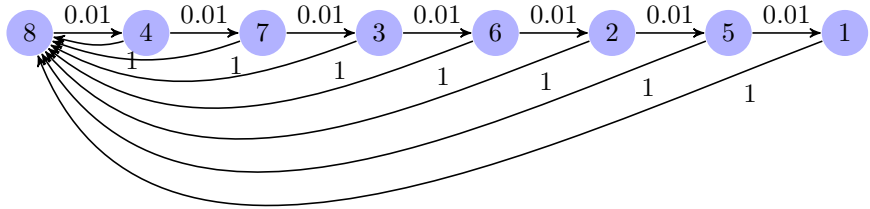 

This process models a fluid queue in which the stationary probabilities have varying orders of magnitudes not due to a single very low transition rate but to a series of moderately small ones: for instance, several unreliable backup systems failing one after the other. By varying the parameter $\kappa$ across several orders of magnitude, we alter the fluid rate in the most unlikely state.

In Figure~\ref{fig:testplot}, we plot the magnitude of the two error measures $e_{\mathrm{norm}}$ and $e_{\mathrm{cw}}$ for the same algorithms as the previous example, for different values of $\kappa$.
\begin{figure}[h!]
\caption{Numerical results for the cascading process} \label{fig:testplot}
\pgfplotstableread{results2014.dat}{\results}
\begin{tikzpicture}
 \begin{axis}[xmode=log,ymode=log,width=\textwidth, legend pos = north west, xlabel=Fluid rate $\kappa$ in the base state $8$, ylabel=Errors on $\Psi$]
  \addplot[color=red,dashdotdotted,mark=*] table[x=M,y=vanilla_error] {\results}; \addlegendentry{GLX-SDA,$e_{\mathrm{norm}}$}
  \addplot[color=red,dashdotdotted,mark=square*] table[x=M,y=cw_vanilla_error] {\results}; \addlegendentry{GLX-SDA,$e_{\mathrm{cw}}$}
  \addplot[color=green,dashed,mark=*] table[x=M,y=li_error] {\results}; \addlegendentry{XXL-SDA,$e_{\mathrm{norm}}$}
  \addplot[color=green,dashed,mark=square*] table[x=M,y=cw_li_error] {\results}; \addlegendentry{XXL-SDA,$e_{\mathrm{cw}}$}
  \addplot[color=blue,dashdotted,mark=*] table[x=M,y=accurate_error] {\results}; \addlegendentry{comp-SDA,$e_{\mathrm{norm}}$}
  \addplot[color=blue,dashdotted,mark=square*] table[x=M,y=cw_accurate_error] {\results}; \addlegendentry{comp-SDA,$e_{\mathrm{cw}}$}
  \end{axis}
\end{tikzpicture}
\end{figure}
As can be seen from the pictures, in this example XXL-SDA has essentially the same numerical behaviour as the non-componentwise accurate GLX-SDA: both have increasing errors as the fluid rate $\kappa$ in the base state $8$ increases. On the other hand, our new algorithm can solve this problem with accuracy that is close to the machine precision $\mp \approx \numprint{2.2e-16}$. \hfill $\blacksquare$ 

\section*{Acknowledgement}
The first author would like to acknowledge the financial support of the Australian Research Council through the Discovery Grant DP110101663. The second author would like to acknowledge the financial support of Istituto Nazionale di Alta Matematica, and thank N.~Higham and M.~Shao for useful discussions on the accuracy of various algorithms for the computation of the exponential of a $-M$-matrix; in particular, N.~Higham pointed us to the paper~\cite{acf}. Both authors are grateful to helpful comments of N.~Bean. 

\appendix

\section{Doubling and censoring}
\renewcommand{\thesection}{A}
\renewcommand{\thedefn}{A.\arabic{defn}} 

In this section, we derive some results connecting the doubling map $\mathcal{F}$ to the concept of \emph{censoring} (or \emph{stochastic complementation}) in Markov chains, which we shall need in the next section when we prove the numerical accuracy of Algorithm~\ref{algo:cwsda}, our componentwise accurate SDA. In particular, Theorem~\ref{censprop} and Lemma~\ref{lemma:repeatedcensoring} are standard results on censoring that we shall need.
\begin{thm}[\protect{\cite[Theorems~2.1 and~2.2]{meyer}}] \label{censprop}
Consider $\mathcal{S} = \{1, \ldots, m\} = S_1 \cup S_2$, where $S_1 \cap S_2 = \emptyset$, and let $P$ be an $m \times m$ stochastic matrix partitioned as
\begin{align} \label{eqn:Ppart}
 P=\m{P_{11} & P_{12} \\ P_{21} & P_{22}},
\end{align}
 where $P_{\ell m}$ contains the transition probabilities from $S_{\ell}$ to $S_m$. Suppose that $I-P_{11}$ is invertible. Then,
 \begin{itemize}
  \item[(i)] $P':=P_{22}+P_{21}(I-P_{11})^{-1}P_{12}$ is stochastic.
  \item[(ii)] If a row vector $\bs{\xi}=\m{\bs{\xi}_1 & \bs{\xi}_2}\geq \bs{0}$ is such that $\bs{\xi} P=\bs{\xi}$, then $\bs{\xi}_2 P'=\bs{\xi}_2$.
 \end{itemize}
\end{thm}

The matrix $P'$ is said to be the matrix obtained by \emph{censoring} the set $\mathcal{S}_1$ or the submatrix $P_{11}$. Censoring can alternatively be described in fully linear algebraic terms: the matrix $I-P'$ is the Schur complement of the submatrix $I-P_{11}$ in $I-P$.

%
\begin{lem}[Quotient property, \cite{crabha}] \label{lemma:repeatedcensoring}
 Let $P$ be an $m \times m$ stochastic matrix, and $\mathcal{T}$ and $\mathcal{U}$ be two disjoint subsets of $\mathcal{S} = \{1,2,\dots,m\}$, with $\mathcal{S} \backslash (\mathcal{T} \cup \mathcal{U}) \neq \emptyset$. Define $P'$ to be the matrix obtained from $P$ by censoring $\mathcal{T}$, and $P''$ the matrix obtained from $P'$ by censoring $\mathcal{U}$. 
 
Then, $P''$ is also the matrix obtained from $P$ by censoring $\mathcal{T}\cup\mathcal{U}$.
\end{lem}

It was proved in \cite{bmp} that the doubling map $\mathcal{F}$ \eqref{eqn:doubling} corresponds to a step in another, similar iteration, \emph{Cyclic Reduction}, on the matrices
\begin{equation} \label{Am}
 A_+ :=\m{E & 0\\ 0 & 0}, \quad A_= :=\m{0 & G\\ H & 0}, \quad A_-:=\m{0 & 0\\ 0 & F}.
\end{equation}
Moreover, it was also established \cite{bg,blm05} that Cyclic Reduction corresponds to Schur complementation on a suitable tridiagonal block-Toeplitz matrix; thus, it should not be surprising that there is a direct interpretation of doubling as Schur complementation/censoring, which is our next result. This result is slightly different from the existing ones, since we are interested with block-circulant matrices rather than block-tridiagonal ones here; nevertheless, the proof is similar (see, for instance, \cite[Section~7.3]{blm05}). 

Let $Z_m$ be the $m \times m$ \emph{circulant generator matrix}, defined by
\begin{align*} 
 (Z_m)_{ij} = \begin{cases}
              1 & j-i \equiv 1 \mod m,\\
              0 & \text{otherwise},
             \end{cases}
\end{align*}
and denote by $\otimes$ the Kronecker product.
\begin{thm} 
\label{bigcensor}
Let $P$ be a $n\times n$ stochastic matrix, and define $E$, $F$, $G$, and $H$ as in ~\eqref{eqn:ppart}, and $A_+$, $A_=$, and $A_-$ as in \eqref{Am}. Then, for $k \geq 0$ the matrix $\mathcal{F}^k(P)$ is obtained by censoring the top $(2^k-1)n\times (2^k-1)n$ block of the $2^kn\times 2^kn$ matrix
\begin{align}
   P^{(2^k)} &:=Z_{2^k}^{-1}\otimes A_-  + I_{2^k} \otimes A_= + Z_{2^k} \otimes A_+ \label{eqn:P}\\
   &\;= \m{A_= & A_+  & & &  A_-\\
       A_- & A_= & A_+ & & \\
       & A_- & A_= & \ddots & &\\ 
        & & \ddots & \ddots &  A_+\\
        A_+ & & &  A_- & A_=}. \nonumber
\end{align}
\end{thm}
Here and in the following, for clarity we omit  blocks of zeros from some large block matrices such as the one in \eqref{eqn:P}.
\begin{proof}
 We prove the theorem by induction. For $k=1$, 
\begin{align*} 
 P^{(2)}=\m{A_= & A_-+A_+\\ A_-+A_+ & A_=} = \left[ \begin{array}{cc|cc} 0 & G & E & 0\\ H & 0 & 0 & F \\ \hline E & 0 & 0 & G\\ 0 & F & H & 0\end{array}\right].
\end{align*} 
Then, using the identity
\begin{align*} 
 \left(I-\m{0 & G\\ H & 0}\right)^{-1} = \m{I & G\\H & I}\m{(I-GH)^{-1} & 0 \\ 0 & (I-HG)^{-1}}
\end{align*} 
we can verify the result directly. 

Now, we assume that the result holds for a certain $k$ and prove it for $k+1$. Let $\widehat{P}:=\mathcal{F}(P)$. We first prove that the matrix obtained by censoring out the components corresponding to the blocks with odd index numbers from $P^{(2^{k+1})}$ is exactly $\widehat{P}^{(2^k)}$, that is, the matrix with the same structure \eqref{eqn:P} built starting from $\widehat{P}$.
The result then follows from Lemma~\ref{lemma:repeatedcensoring}.

We first reorder the blocks in $P^{(2^{k+1})}$ to put the odd-numbered ones on the top, obtaining
\begin{align*} 
& \left[
 \begin{array}{cccc|cccc}
  A_= & & & & A_+ & &  & A_-\\
  & A_= & & & A_- & A_+ \\
  & & \ddots & & & \ddots & \ddots &\\
   & & &  A_= & & & A_- & A_+ \\ \hline
   A_- & A_+ & & & A_=\\
   & A_- & \ddots  & & &  A_=\\
   & & \ddots & A_+ & & & \ddots \\
   A_+ & & & A_- & & &  &A_=
 \end{array}
\right] \\
& =\left[
\begin{array}{c|c}
 I \otimes A_= & Z^{-1} \otimes A_- + I\otimes A_+\\ \hline
 I \otimes A_- + Z \otimes A_+ & I\otimes A_=
\end{array}
\right].
\end{align*} 
Next, we censor the top part, $I \otimes A_{=}$, to obtain
\begin{align*}
 &\phantom{{}={}} I \otimes A_= + (I \otimes A_- + Z \otimes A_+)(I\otimes (I - A_=))^{-1}(Z^{-1} \otimes A_- + I\otimes A_+)\\
 &=Z^{-1} \otimes A_-(I-A_=)^{-1}A_- +  I \otimes (A_= + A_-(I-A_=)^{-1}A_+ \\  
 & \quad\quad  +A_+(I-A_=)^{-1}A_-) + Z \otimes A_+(I-A_=)^{-1}A_+\\
 & = Z^{-1} \otimes \m{0 & 0\\ 0 & \widehat{F}} + I \otimes \m{0 & \widehat{G}\\ \widehat{H} & 0} + Z \otimes \m{\widehat{E} & 0\\ 0 & 0},
\end{align*}
where $\widehat{E},\widehat{F},\widehat{G},$ and $\widehat{H}$ are the blocks of $\widehat{P}$ as in \eqref{eqn:SDA}. The last expression is precisely $\widehat{P}^{(2^k)}$. By the inductive hypothesis, the matrix obtained by censoring the top $(2^k-1)n \times (2^k-1)n$ of this matrix is $\mathcal{F}^k(\widehat{P})=\mathcal{F}^{k+1}(P)$, which completes the proof.
\end{proof}

\begin{rem}
 One can obtain a different, probabilistic proof for this result relying on a physical interpretation of doubling algorithms in terms of the underlying fluid queue. We are currently working on an explaination for these algorithms from a probabilistic point of view, giving some insight that complements the linear-algebraic point of view presented in this paper.
\end{rem}

Similarly, we can express the initial values of the doubling algorithms as the censoring of a suitable stochastic matrix.
\begin{thm}
 Let $(T,C)$ be the transition and rate matrices of a fluid queue, and let $\gamma \leq \left(\max_{i} Q_{ii}\right)^{-1}$. Then, $P_0$ as in \eqref{eqn:P0} is obtained by censoring the first $n$ states out of
 \begin{equation} \label{eqn:S}
  S=\m{I-\gamma Q & \gamma R\\ I & 0}.
 \end{equation}
\end{thm}
\begin{proof}
 Clearly, $0+I(I-(I-\gamma Q))^{-1}\gamma R=Q^{-1}R$, so it suffices to prove that $S$ is stochastic. For the nonnegativity, we need only to verify that the diagonal entries of $R$ are nonnegative, which holds because $\alpha\leq \alpha_{\mathrm{opt}},\beta\leq \beta_{\mathrm{opt}}$, and that the diagonal entries of $I-\gamma Q$ are nonnegative, which holds because of the choice of $\gamma$. Moreover, substituting the definitions of $Q$ and $R$ from \eqref{eqn:P0}, and using $T\ones=\bs{0}$, one sees that $S\ones =\ones$.
\end{proof}
\begin{cor} \label{Pstoc}
 Let $\bs{\xi}\geq 0$ be such that $\bs{\xi}T=\bs{0}$. Then, for $k \geq 0$, $P_k$ is stochastic and $\bs{\xi}\abs{C}P_k= \bs{\xi}\abs{C}$.
\end{cor}
\begin{proof}
By their definitions and the fact that $\bs{\xi}T=\bs{0}$, 
\begin{align*}
\bs{\xi} Q =\bs{\xi}\abs{C}=\bs{\xi} R,
\end{align*}
and consequently that $\m{\bs{\xi} & \gamma\bs{\xi}\abs{C}}S=\m{\bs{\xi} & \gamma\bs{\xi}\abs{C}}$. Hence, by applying Theorem~\ref{censprop} to $P_0$ as a censoring of $S$, the corollary holds for $k=0$.

For $k > 0$, one can verify directly that $(\ones^\top\otimes \bs{\xi})P^{(2^k)}=(\ones^\top\otimes \bs{\xi})$, so the result follows again from applying Theorem~\ref{censprop} to $P_k$ as a censoring of $P^{(2^k)}$.
\end{proof}

\renewcommand{\thesection}{B}
\renewcommand{\thedefn}{B.\arabic{defn}} 
\section{Componentwise stability of SDA}

In this section, we prove the results on componentwise stability of SDA, thus giving the proof of Theorem~\ref{Gtrick-above}. We begin with a perturbation bound that tells us how the iterates are affected by a small change in the initial values.

\subsection{Componentwise perturbation bounds}

We focus on first-order results, assuming $\varepsilon$ to be a small parameter and hence ignoring all terms containing $\varepsilon^2$ and higher powers of $\varepsilon$. For brevity, we use the notation $A \leqdot B$ to denote $A \leq [1+O(\varepsilon^2)] B$. Moreover, we focus at first on the case $n_+\approx n_-$; hence, we replace $n_+$ and $n_-$ liberally with $n_{\max}:=\max(n_+,n_-)$ in our bound. The results of a more accurate analysis that keeps track of the differences between $n_+$ and $n_-$ are reported later in Section~\ref{sec:unbalanced}.

We state here two useful lemmas. The first is a small variant of \cite[Lemma~2.2]{AlfXY02}, while the second is a simple consequence of the triangle inequality.
\begin{lem} \label{lemma:pertinv}
 Let $(\offdiag(A),\ones, \bs{w})$ and $(\offdiag(\tilde{A}),\ones,$ $\tilde{\bs{w}})$ be triplet representations for two $m\times m$ M-matrices $A$ and $\tilde{A}$ such that 
$$\abs{\offdiag(A)-\offdiag(\tilde{A})}\leq \varepsilon \abs{\offdiag(A)} \quad \mbox{and} \quad \abs{\bs{w}-\tilde{\bs{w}}} \leq \varepsilon \bs{w}.$$ 
Then, 
\begin{align} 
\abs{A^{-1}-\tilde{A}^{-1}} \leqdot (2m-1)\varepsilon A^{-1}.
\end{align} 
\end{lem}

\begin{lem} 
	\label{sumproduct}
 Let $A,B,\tilde{A},$ and $\tilde{B}$ be nonnegative matrices of suitable sizes such that $\abs{A-\tilde{A}}\leqdot a\varepsilon A$ and $\abs{B-\tilde{B}}\leqdot b\varepsilon B$. Then,
 \begin{itemize}
  \item[(i)] $\abs{A+B-(\tilde{A}+\tilde{B})} \leqdot \max(a,b) \varepsilon (A+B)$.
  \item[(ii)] $\abs{AB-\tilde{A}\tilde{B}} \leqdot (a+b)\varepsilon AB$.
 \end{itemize}
\end{lem}

The following lemma provides a perturbation bound for censoring.
\begin{lem} 
\label{lem:censorbound}
Consider two stochastic matrices 
\begin{align*} 
  P=\m{P_{11} & P_{12}\\ P_{21} & P_{22}} \text{ and } \tilde{P}=\m{\tilde{P}_{11} & \tilde{P}_{12} \\ \tilde{P}_{21} & \tilde{P}_{22}},
 \end{align*}
with $P_{11},\tilde{P}_{11}\in\mathbb{R}^{m\times m}$, $\abs{\offdiag(P-\tilde{P})}\leq \varepsilon \offdiag(P),$ and suppose that $I-P_{11}$ is invertible. Then,
\begin{align*}
\abs{\tilde{P}_{21}(I-\tilde{P}_{11})^{-1}\tilde{P}_{12}-P_{21}(I-P_{11})^{-1}P_{12}} \leqdot (2m+1)\varepsilon P_{21}(I-P_{11})^{-1}P_{12}.
\end{align*} 
If, in addition, $\abs{\diag(P_{22})-\diag(\tilde{P}_{22})}\leq (2m+1)\varepsilon P_{22}$, then 
\begin{multline*}
\abs{\tilde{P}_{22}+\tilde{P}_{21}(I-\tilde{P}_{11})^{-1}\tilde{P}_{12}-P_{22}-P_{21}(I-P_{11})^{-1}P_{12}} \\ \leqdot (2m+1)\varepsilon (P_{22}+P_{21}(I-P_{11})^{-1}P_{12}).  
 \end{multline*} 
\end{lem}
\begin{proof}
We have 
\begin{align*} 
\abs{(I-P_{11})\ones-(I-\tilde{P}_{11})\ones}=\abs{P_{12}\ones-\tilde{P}_{12}\ones}\leq \varepsilon P_{12}\ones.
\end{align*}
Thus, $I-P_{11}$ and $I-\tilde{P}_{11}$ satisfy the hypotheses of Lemma~\ref{lemma:pertinv} and hence $\abs{(I-\tilde{P}_{11})^{-1}-(I-P_{11})^{-1}}\leq (2m-1)\varepsilon (I-P_{11})^{-1}$. The results now follow by propagating errors with the help of Lemma~\ref{sumproduct}.
\end{proof}

We start from a componentwise perturbation bound for $P_0$.
\begin{lem} 
\label{lemma:initbound}
Let $(T,C)$ and $(\tilde{T},\tilde{C})$ be two pairs of the transition matrix and the rate matrix of two fluid queues such that 
\begin{align*}
\abs{\offdiag(T)-\offdiag(\tilde{T})}\leq \varepsilon \hspace*{-0.05cm}\offdiag(T) \quad \mbox{and} \quad \abs{C-\tilde{C}}\leq \varepsilon \abs{C}.
\end{align*}
Consider $P_0=Q^{-1}R$ and its analogous $\tilde{P}_0=\tilde{Q}^{-1}\tilde{R}$ computed using the same formulas \eqref{eqn:P0} but starting from $\tilde{T}$ and $\tilde{C}$, and with the same values of $\alpha$ and $\beta$, chosen such that $\alpha \leq \alpha_{\mathrm{opt}},\beta\leq \beta_{\mathrm{opt}}$.
Then, 
\begin{align*} 
\abs{P_0-\tilde{P}_0} \leqdot (2n+1) \varepsilon P_0.
\end{align*} 
\end{lem}
\begin{proof}
We have
\begin{align*}
 \abs{\tilde{R}_{ii}-R_{ii}} &=
 \abs{\tilde{C}_+ + \beta \tilde{T}_{ii}  - C_+ - \beta T_{ii}} \\
& \leq \abs{\tilde{C}_+ -C_+} + \beta \abs{\tilde{T_{ii}} - T_{ii}} \\
 & \leq \varepsilon C_+ + \beta \sum_{j\neq i} \abs{\tilde{T}_{ij}-T_{ij}} \\
 & \leq \varepsilon C_+ + \beta \varepsilon \sum_{j\neq i} T_{ij} \\
 & =\varepsilon R_{ii}.
\end{align*}
Hence, for $S$ from \eqref{eqn:S} and its analog $\tilde{S}$ (where we take the same value of $\gamma$ for both), we have the bound $\abs{\offdiag(\tilde{S})-\offdiag(S)}\leq \varepsilon \offdiag(S)$. By Lemma~\ref{lem:censorbound}, we obtain the required result.
\end{proof}
Next, we derive a perturbation bound for the result of $k$ steps of SDA.
\begin{lem} \label{pertmultsda}
 Let $P$ and $\tilde{P}$ be two $n\times n$ stochastic matrices such that $\abs{\tilde{P}-P}\leq \varepsilon P$, and $\mathcal{F}$ be the doubling map~\eqref{eqn:doubling}. Then,
\begin{align*} 
  \abs{\mathcal{F}^k(\tilde{P})-\mathcal{F}^k(P)}\leqdot n 2^k \varepsilon \mathcal{F}^k(P).
 \end{align*} 
\end{lem}
\begin{proof}
Consider the matrices $P^{(2^k)}$ and its analog $\tilde{P}^{(2^k)}$ defined in the same way but starting from $\tilde{P}$. They are stochastic matrices, and they satisfy the hypotheses of Lemma~\ref{lem:censorbound}. We apply the second part of this lemma, which completes the proof. 
\end{proof}
By applying the first part of Lemma~\ref{lem:censorbound} rather than the second, we get a corresponding bound for $$J_k:=G_{k}-G_{k-1}=E_k(I-G_{k-1}H_{k-1})^{-1}G_{k-1}F_{k-1} \geq 0.$$
\begin{cor} \label{jay}
 With the hypotheses of Lemma~\ref{pertmultsda}, we have
 \[
  \abs{\tilde{J}_{k}- J_{k}} \leqdot n 2^k\varepsilon J_k.
 \]
\end{cor}

The above bound gives us immediately $\abs{\tilde{G}_k-G_k}\leqdot n2^k\varepsilon G_k$ for the doubling iterates starting from two nearby sets of matrices. This bound, however, is quite weak: in particular, it would allow for $\lim_{k\to\infty} \abs{\tilde{G}_k-G_k}=\infty$, while one would expect this quantity to converge, since $G_k\to \Psi$ and $\tilde{G}_k\to\tilde{\Psi}$.

Here, we overestimate the error: essentially, we compute a new term of the sum
$G_k = G_0 + J_1 + J_2 + \dots + J_k$ at each step, and at the same time make the following estimate:
\begin{align}
 \abs{\tilde{G}_k-G_k} &\leq \abs{\tilde{G}_0-G_0} + \abs{\tilde{J}_1-J_1} + \abs{\tilde{J}_1-J_1} + \abs{\tilde{J}_2-J_2} + \dots + \abs{\tilde{J}_k-J_k} \nonumber\\ 
 & \leqdot \varepsilon G_0 + n\varepsilon J_1 + 2n \varepsilon J_2 + \dots + 2^kn \varepsilon J_k \nonumber\\
 & \leq 2^kn(G_0+J_1+J_2+\dots + J_k) \nonumber.
\end{align}
Indeed, the last line is how errors in sums are bounded using Lemma~\ref{sumproduct}. However, the factors $J_k$ decay exponentially like $O(\delta^{2^k})$, and this compensates for the growing powers of two. The following theorem formalizes this intuition to improve this bound.
\begin{thm} \label{gtrick_pert}
 With the hypotheses of Lemma~\ref{pertmultsda}, we have
 \[
  \abs{\tilde{G}_k-G_k} \leqdot \left(1 + \frac{4\widehat{K}_0n}{1-\delta}\right) \varepsilon \Psi.
 \]
 for a suitable constant $\widehat{K}_0$.
\end{thm}
\begin{proof}
We know from Theorem~\ref{thm:sdaprops} that $0\leq \Psi-G_k\leq K_0 \delta^{2^k}$, and that $\Psi-G_k\leq \Psi$. Hence we can write, for a suitable constant $\widehat{K}_0$ (different in general from $K_0$), $\Psi-G_k \leq \widehat{K}_0 \delta^{2^k}\Psi$.
\begin{align*}
\abs{\tilde{G}_k-G_k} & \leq \abs{\tilde{G}_0-G_0}+ \sum_{h=1}^k \abs{\tilde{J}_{h} - J_{h}} \\
 & \leqdot \varepsilon G_0 + n \varepsilon \sum_{h=1}^k 2^h J_h \\
 & \leq \varepsilon G_0 + n \varepsilon \sum_{h=1}^k 2^h (\Psi-G_{h-1}) \\
 & \leq \varepsilon G_0 + n \varepsilon \sum_{h=1}^k 2^h \widehat{K}_0 \delta^{2^{h-1}}\Psi \\
 & \leq \varepsilon G_0 + 2\widehat{K}_0n \varepsilon \sum_{h=1}^k 2^{h-1} \delta^{2^{h-1}}\Psi \\
 & \leq \varepsilon G_0 + 4\widehat{K}_0n \varepsilon \sum_{j=1}^{2^k} \delta^{j} \Psi 
  \leq \varepsilon \Psi + \frac{4\widehat{K}_0n}{1-\delta} \varepsilon \Psi \qedhere
\end{align*} 
\end{proof}
Clearly, an analogous result for $H_k$ holds.

\subsection{Numerical stability}

We now establish a more complex result that keeps track also of the numerical errors in an implementation of the algorithm in floating-point arithmetic. We denote by ${\mp}$ the machine precision, and ignore second-order terms in ${\mp}$ by hiding them again in the syntax $A \leqdot B$.

We shall rely on the two following lemmas. The first one is a version of Lemma~\ref{sumproduct} that keeps track of floating point arithmetic errors as well.
\begin{lem} 
		\label{numsumprod}
 Let $A$ and $B$ be nonnegative matrices of suitable sizes, and let $\tilde{A}$ and $\tilde{B}$ be matrices of floating point numbers such that $\abs{A-\tilde{A}}\leqdot a {\mp} A$, $\abs{B-\tilde{B}}\leqdot b {\mp} B$.
 \begin{enumerate}
  \item Let $\tilde{S}$ be the result of the matrix addition $\tilde{A}+\tilde{B}$ performed in floating point arithmetic. Then,
\[
 \abs{\tilde{S}-(A+B)}\leqdot (\max(a,b)+1){\mp}(A+B).
\]
 \item Let $\tilde{P}$ be the result of the matrix product $\tilde{A}\tilde{B}$ performed in floating point arithmetic. Then,
\[
 \abs{\tilde{P}-(AB)}\leqdot (a+b+m){\mp}(AB),
\]
where $m$ is the number of columns of $A$ (and rows of $B$).
 \end{enumerate}
\end{lem}
The second one is a slight variation of Theorem~\ref{th:gthstab} to convert it to the same form as the bounds in Lemma~\ref{numsumprod}.
\begin{lem}
	 \label{numinv2}
Let $A, \tilde{A} \in\mathbb{R}^{m\times m}$ be two M-matrices with triplet representations $(\offdiag(A),\ones,\bs{w})$ and $(\offdiag(\tilde{A}),\ones,\tilde{\bs{w}})$, and $B, \tilde{B} \in\mathbb{R}^{m\times k}$ be nonnegative, such that
\begin{align*} 
\abs{\offdiag(A)-\offdiag(\tilde{A})} & \leqdot a{\mp}\abs{\offdiag(A)}, \\
\abs{B-\tilde{B}} & \leqdot b{\mp}B, \\
\abs{\bs{w}-\tilde{\bs{w}}} & \leqdot a{\mp}\bs{w},
\end{align*} 
and all of $\offdiag(\tilde{A}), \tilde{B},$ and $\tilde{\bs{w}}$ are floating-point numbers. 

Let $\tilde{X}$ be the result of the operation $\tilde{A}^{-1}\tilde{B}$ performed in floating-point arithmetic with the GTH-like algorithm. Then,
\[
 \abs{\tilde{X}-A^{-1}B} \leqdot ((2m-1)a+b+\psi(m)){\mp} A^{-1}B.
\]
\end{lem}
\begin{proof}
 By Lemma~\ref{lemma:pertinv}, $\abs{A^{-1}-\tilde{A}^{-1}}\leqdot (2m-1){\mp}A^{-1}$. Using this result and Lemma~\ref{sumproduct}, we get 
\begin{align*} 
\abs{\tilde{A}^{-1}\tilde{B}-A^{-1}B} & \leqdot ((2m-1)a+b){\mp} A^{-1}B.
\intertext{By Theorem~\ref{th:gthstab},} 
\abs{\tilde{X}-\tilde{A}^{-1}\tilde{B}} & \leqdot \psi(m){\mp}\tilde{A}^{-1}\tilde{B}\leqdot \psi(m){\mp}A^{-1}B.
\end{align*}
The result follows using the triangle inequality.
\end{proof}

Our first result towards an accuracy proof concerns the initial values.
\begin{prop} \label{numinitial}
 Let Algorithm~\ref{algo:cwsda} be performed in floating-point arithmetic with $\alpha\leq 2\alpha_\mathrm{opt}$ and $\beta\leq 2\beta_\mathrm{opt}$. Then, the computed approximation $\tilde{P}_0$ to $P_0$ satisfies the bound
 \[
  \abs{\tilde{P}_0-P_0} \leqdot K_1{\mp}P_0,
 \]
 with $K_1:=\psi(n)+5n+2$ and $\psi(n)$ defined as in Theorem~\ref{theo:fourone}.
\end{prop}
\begin{proof}
We assume without loss of generality that $\alpha$ and $\beta$ are exact machine numbers (otherwise, we simply replace them with their machine representation). The machine representations of $T$ and $c_i$ satisfy $\abs{\tilde{T}-T}\leq \mp T$ and $\abs{\tilde{c}_i-c_i}\leq \mp c_i$.  

Let $\tilde{Q}$ and $\tilde{R}$ be the computed versions of $Q$ and $R$ in floating-point arithmetic. They are computed by taking the machine representation of the corresponding entry of $T$ and multiplying it by $\alpha$ or $\beta$, hence 
\begin{align*} 
\abs{\offdiag(\tilde{Q})-\offdiag(Q)} & \leqdot 2\mp \offdiag(Q) 
\intertext{and}
\abs{\offdiag(\tilde{R})-\offdiag(R)} & \leqdot 2\mp \offdiag(Q).
\end{align*} 
The quantities $d_i:=\beta \sum_{j\neq i} T_{ij}$, for $i\in \mathcal{S}_+$, are computed (with $n-1$ additions and one multiplication) up to an error $\abs{\tilde{d}_i-d_i}\leqdot (n+1) \mp d_i$, as we get by applying repeatedly Lemma~\ref{numsumprod}. Moreover, due to our choice of $\beta$, $c_i \geq 2d_i$, which implies $d_i\leq c_i-d_i$ and $c_i\leq 2(c_i-d_i)$. Hence we have
\begin{align*}
\abs{\operatorname{fl}(\tilde{c}_i-\tilde{d}_i) - (c_i-d_i)} &= \abs{(\tilde{c}_i-\tilde{d}_i)(1+\varepsilon)-(c_i-d_i)}\\ &\leqdot \abs{\tilde{c}_i-c_i}+\abs{\tilde{d}_i-d_i} + \mp \abs{\tilde{c}_i-\tilde{d}_i}\\
& \leqdot \mp c_i + (n+1)\mp d_i + \mp (c_i-d_i)\\
& \leq (2+n+1+1)\mp(c_i-d_i)=(n+4)\mp (c_i-d_i).
\end{align*}
A similar bound applies for $i\in\mathcal{S}_-$. Putting everything together, we have $\abs{\tilde{R}-R}\leq (n+4)R$, $\abs{\offdiag(\tilde{Q})-\offdiag(Q)}\leqdot 2\mp \offdiag(Q)$, and $\abs{\tilde{C}\ones-C\ones}\leq \mp C\ones$. Hence, by Lemma~\ref{numinv2}
\[
 \abs{\tilde{P}_0-P_0}\leq (2(2n-1)+(n+4)+\psi(n))\mp P_0.
\]
\end{proof}

We are now ready to assess the numerical error a single SDA step.

\begin{lem} \label{steperror}
Let $\tilde{\mathcal{F}}$ be the map that computes approximately the result of one iteration of SDA implementing Equations~\eqref{eqn:SDA} in machine arithmetic. For any $n\times n$ stochastic matrix $P$, we have
\[
 \abs{\tilde{\mathcal{F}}(P)-\mathcal{F}(P)} \leqdot K_2\mp \mathcal{F}(P),
\]
with $K_2:=\psi(n_{\max}) + 4n_{\max}^2+2n_{\max}-1$
\end{lem}
\begin{proof}
Let $E,F,G,H$ be the blocks of $P$, $\widehat{E},\widehat{F},\widehat{G},\widehat{H}$ those of $\mathcal{F}(P)$, and $\tilde{E},\tilde{F},\tilde{G},\tilde{H}$ those of $\tilde{\mathcal{F}}(P)$. For simplicity, with a slight abuse of notation, for any matrix $X\geq 0$ appearing in the algorithm we denote here by $c(X)$ its computed value in machine arithmetic along the SDA step computation.
Applying Lemma~\ref{numsumprod} repeatedly, we have $\abs{c(GH)-GH} \leq n_-{\mp} \tilde{G}\tilde{H}$ and
\begin{align*}
 \abs{c(F\ones) - F\ones} & \leqdot n_-{\mp} F\ones,\\
 \abs{c(GF\ones) - GF\ones} & \leqdot 2n_- {\mp} GF\ones,\\
 \abs{c(E\ones) - E\ones} & \leqdot n_+{\mp} E\ones,\\
 \abs{c(GF\ones+E\ones) - GF\ones - E\ones} & \leqdot (2n_{\max}+1){\mp} (GF\ones - E\ones).\\ 
\end{align*}
Hence
\begin{multline*}
\abs{c((I-GH)^{-1}G)-(I-GH)^{-1}G} \\ \leqdot ((2n_{\max}-1)(2n_{\max}+1) + \psi(n_{\max})) \mp (I-GH)^{-1}G,
\end{multline*}
and carrying on the error propagation in the same fashion leads to
\begin{align*}
 \abs{\tilde{G}-\widehat{G}} &= \abs{c(G+E(I-GH)^{-1}GF)- G - E(I-GH)^{-1}GF} \\
 & \leqdot (\psi(n_{\max}) + 4n_{\max}^2+2n_{\max}-1) \mp \widehat{G}.
\end{align*}
The error computations for the other three iterates $\tilde{E}$, $\tilde{F}$, $\tilde{H}$ are analogous.
\end{proof}

We are now ready to tackle several doubling steps together.
\begin{thm} \label{finalbound}
 Let $\tilde{P}_k=\tilde{\mathcal{F}}^k(\tilde{P}_0)$ be the approximation to $P_k=\mathcal{F}^k(P_0)$ computed by Algorithm~\ref{algo:cwsda} performed in floating-point arithmetic with machine precision ${\mp}$, using values of the parameters $\alpha,\beta$ such that $\alpha\leq 2\alpha_\mathrm{opt}$ and $\beta\leq 2\beta_\mathrm{opt}$. Then,
 \[
  \abs{\tilde{P}_k-P_k} \leq n 2^k (K_1+K_2)\mp P_{k}.
 \]
\end{thm}
\begin{proof}
We prove the result by induction on $k$. 

The following manipulation is nothing but a formal version of the statement that when considering first-order error bounds, we can simply add up the local errors in every step of the algorithm. We start from the telescopic sum
\begin{align}
 \abs{\tilde{\mathcal{F}}^k(\tilde{P}_0)-\mathcal{F}^k(\tilde{P}_0)} & \leq \sum_{h=1}^k \abs{\mathcal{F}^{h-1}\tilde{\mathcal{F}}^{k-h+1}(\tilde{P}_0)-\mathcal{F}^{h}\tilde{\mathcal{F}}^{k-h}(\tilde{P}_0)} \nonumber\\
 & = \sum_{h=1}^k \abs{\mathcal{F}^{h-1}\tilde{\mathcal{F}}(\tilde{P}_{k-h})-\mathcal{F}^{h-1}\mathcal{F}(\tilde{P}_{k-h})}. \label{inserthere}
\end{align}
By Lemma~\ref{steperror},
\[
 \abs{\tilde{\mathcal{F}}(\tilde{P}_{k-h}) - \mathcal{F}(\tilde{P}_{k-h})} \leqdot K_2\mp \mathcal{F}(\tilde{P}_{k-h});
\]
hence we can apply Lemma~\ref{pertmultsda} with $\varepsilon=K_2\mp$, to obtain for each $h \geq 1$
\begin{align*}
 \abs{\mathcal{F}^{h-1}\tilde{\mathcal{F}}(\tilde{P}_{k-h})-\mathcal{F}^{h-1}\mathcal{F}(\tilde{P}_{k-h})} & \leqdot n2^{h-1} K_2 \mp \mathcal{F}^h(\tilde{P}_{k-h})\\
 & \leqdot n2^{h-1} K_2 \mp \mathcal{F}^k(P_0).
\end{align*}
In the last line, we have substituted $P_{k-h}$ for $\tilde{P}_{k-h}$, which is possible because the two expressions differ only for a term of magnitude $O(\mp)$ by inductive hypothesis.
 
Insert this inequality into \eqref{inserthere} to obtain
\begin{equation}
\begin{aligned} \label{boundone}
 \abs{\tilde{\mathcal{F}}^k(\tilde{P}_0)-\mathcal{F}^k(\tilde{P}_0)} & \leqdot \sum_{h=1}^k n2^{h-1} K_2\mp \mathcal{F}^k(P_0)\\
  & < n 2^k K_2\mp \mathcal{F}^k(P_0).
\end{aligned}
\end{equation}
Moreover, using again Lemma~\ref{pertmultsda} we get
\begin{equation} \label{boundtwo}
 \abs{\mathcal{F}^k(\tilde{P}_0)-\mathcal{F}^k(P_0)} \leq n2^k K_1\mp \mathcal{F}^k(P_0).
\end{equation}
Combining \eqref{boundone} and~\eqref{boundtwo} we get the required result.
\end{proof}

Again, this error bound is exponential in $k$, which is quite unwelcome. We can improve it using the same trick as in Theorem~\ref{gtrick_pert}. Since computing $J_k$ is just removing the final addition from the algorithm, the previous bound can be easily modified to yield $\abs{\tilde{J}_k-J_k} \leq n2^k(K_1+K_2)\mp J_k$. Now we proceed as in the proof of Theorem~\ref{gtrick_pert} to prove Theorem~\ref{Gtrick-above}. 
%
\begin{proof} \emph{(of Theorem~\ref{Gtrick-above})}
The matrix $\tilde{G}_k$ is the result of the addition (in machine arithmetic) of the matrices $\tilde{G}_0, \tilde{J}_1, \tilde{J}_2, \dots, \tilde{J}_k$, performed in this order at different steps along the algorithm. For this addition, we have
\begin{align*}
 \abs{\tilde{G}_k-G_k} & \leqdot \abs{(\tilde{G}_0 + \sum_{h=1}^k \tilde{J}_h)(1+k\mp)  - (G_0 + \sum_{h=1}^k J_h)} \\
 & \leqdot k\mp G_k + \abs{\tilde{G}_0-G_0}+\sum_{h=1}^k \abs{\tilde{J}_h-J_h} \\
 & \leqdot k\mp G_k + K_1\mp G_0+\sum_{h=1}^k n2^h (K_1+K_2)\mp J_h \\
 & \leq k\mp G_k + K_1\mp G_0+(K_1+K_2)n\sum_{h=1}^k 2^h \mp \tilde{K}_0\delta^{2^h}\Psi \\
 & \leq k\mp \Psi + K_1\mp \Psi+  \frac{2(K_1+K_2)n}{1-\delta}\tilde{K}_0\Psi.
\end{align*} 
\end{proof}

\subsection{The unbalanced case} \label{sec:unbalanced}
In the previous bounds we have replaced $n_+$ and $n_-$ with $n_{\max}$ for simplicity; hence, these bounds are overestimated whenever $n_-\neq n_+$. It turns out that there is a different strategy that we can adopt in this case, which leads to a slightly smaller value of $K_2$. First of all, we recall from \cite{xxl12} that the following alternative formulas equivalent to \eqref{eqn:SDA} can also be used for the doubling step
\begin{subequations} \label{eqn:SDAalt}
\begin{align}
E_{k + 1} & := E_k(I+G_k(I - H_k G_k)^{-1}H_k)E_k, \label{eqn:SDAalt_Ek} \\
F_{k + 1} & := F_k(I+H_k(I - G_k H_k)^{-1}G_k)F_k, \label{eqn:SDAalt_Fk} \\
G_{k + 1} & := G_k + E_k G_k(I - H_kG_k)^{-1}F_k, \label{eqn:SDAalt_Gk}\\ 
H_{k + 1} & := H_k + F_kH_k(I - G_kH_k)^{-1}E_k.  \label{eqn:SDAalt_Hk}
\end{align}
\end{subequations}
If $n_+=n_-$, the formulas \eqref{eqn:SDA} are the ones with the lower computational cost; however, if $n_+ < n_-$, one can invert the $n_+\times n_+$ matrix $I-G_k H_k$ rather than the larger $n_-\times n_-$ matrix $I-H_kG_k$, that is, use \eqref{eqn:SDA_Ek}, \eqref{eqn:SDAalt_Fk}, \eqref{eqn:SDA_Gk}, \eqref{eqn:SDAalt_Hk}, and vice versa if $n_- < n_+$.

By using the appropriate formulas, one can derive a slightly tighter version of Lemma~\ref{steperror}. We report only the final result here, since the proof is basically unchanged.
\begin{thm}
 If we choose for each iterate the formula with the smallest matrix to invert among \eqref{eqn:SDA} and \eqref{eqn:SDAalt}, then in Lemma~\ref{steperror} (as well as in Theorems~\ref{finalbound} and~\ref{Gtrick-above}) we may replace the factor $K_2$ with
\begin{align}
\tilde{K}_2 &:=\psi(n_{\min})+(2n_{\max}+1)(2n_{\min}-1) +2n_{\max}+n_{\min} \nonumber  \\
&\;= \psi(n_{\min})+4n_{\max}n_{\min} + 3n_{\min}. \nonumber
 \end{align} 
\end{thm}

\bibliographystyle{abbrv}
\bibliography{NP2013}

\end{document}